\documentclass[11pt]{article}
\usepackage{subfigure}
\usepackage{amssymb,amsmath,amsthm,epsfig,a4,verbatim,pstricks,ifthen,setspace,ulem,color} 
\usepackage{url}
\usepackage{float}
\allowdisplaybreaks
\definecolor{mygreen}{rgb}{0.05,0.6,0.05}
\newtheorem{thm}{\sc Theorem.}[section]
\newtheorem{lem}{\sc Lemma.}[section]

\newenvironment{AMS}%
{{\upshape\bfseries AMS subject classifications. }\ignorespaces}{}
\newenvironment{keywords}{{\upshape\bfseries Key words. }\ignorespaces}{}

\newcommand{\bR}{\mathbb{R}}

\def\ba{\begin{align}}
\def\ea{\end{align}}

\def\epsilon{\varepsilon} 

\def\hat{\widehat}

\begin{document}
\title{Finite element error analysis for a system coupling surface evolution to diffusion on the surface}
\author{Klaus Deckelnick  \footnotemark[3]\ \and 
        Vanessa Styles \footnotemark[4]}

\renewcommand{\thefootnote}{\fnsymbol{footnote}}
\footnotetext[3] {Institut f{\"u}r Analysis und Numerik,
 Otto-von-Guericke-Universit{\"a}t Magdeburg, 39106 
Magdeburg, Germany}
\footnotetext[4]{Department of Mathematics, University of Sussex, Brighton, BN1 9RF, UK}

\date{}

\maketitle

\begin{abstract}
We consider a numerical scheme for the  approximation of a system that couples the evolution of a two--dimensional 
hypersurface to a reaction--diffusion equation on the surface. The surfaces are assumed to be graphs and
evolve according to forced mean curvature flow. The method uses continuous, piecewise linear finite elements in
space and a backward Euler scheme in time. Assuming the existence of a smooth solution we prove optimal error bounds 
both in $L^\infty(L^2)$ and in $L^2(H^1)$. We present several numerical experiments that confirm our theoretical
findings and apply the method in order to simulate diffusion induced grain boundary motion.
\end{abstract} 


\renewcommand{\thefootnote}{\arabic{footnote}}

\section{Introduction} \label{sec:1}

\noindent
In this paper we analyse a finite element scheme for approximating a system which couples diffusion on a surface to an equation that determines the evolution of the surface.
More precisely, we want to find a family of surfaces $(\Gamma(t))_{t \in [0,T]} \subset \mathbb R^3$ and a 
function $w: \bigcup_{t \in [0,T]} \bigl( \Gamma(t) \times \lbrace t \rbrace \bigr) \rightarrow \bR$ such that
\begin{subequations}
\begin{alignat}{2}
V & = H + f(w)   \qquad &&\mbox{on }\Gamma(t), \quad t \in (0,T], \label{Gamma1}\\
\partial^\bullet w  & = \Delta_{\Gamma} w + H \,V\,w + g(V,w)  \qquad &&
\mbox{on }\Gamma(t), \quad t \in (0,T]. \label{Gamma2}
\end{alignat}
\end{subequations}
Here, $V$ and $H$ are the normal velocity and the mean curvature of $\Gamma(t)$ corresponding to
the choice $\nu$ of a unit normal, while $\Delta_{\Gamma}$ denotes the Laplace--Beltrami operator on  $\Gamma(t)$. Furthermore,
$\partial^\bullet w = w_t + V\, \frac{\partial w}{\partial \nu}$ is the material derivative of $w$ and $f:\mathbb R \rightarrow \mathbb R, ~g:\mathbb R^2 \rightarrow \mathbb R$
are given functions. We are particularly interested
in surfaces $\Gamma(t)$ which can be represented as the graph of a function $u: \bar \Omega \times [0,T] \rightarrow \mathbb R$, i.e.
\begin{equation} \label{graph}
\Gamma(t) = \lbrace (x,u(x,t)) \in \mathbb R^3 \, | \, x \in \bar \Omega \rbrace,
\end{equation}
where $\Omega \subset \mathbb R^2$ is a bounded domain with a smooth boundary. Thus, $(\Gamma(t))_{t \in [0,T]}$ is a family of surfaces with boundary, which evolves
according to forced mean curvature flow 
in the cylindrical set $A= \bar \Omega \times \mathbb R$. In what follows we consider the following boundary conditions:
\begin{subequations}
\begin{alignat}{2}
\nu \cdot \nu_{\partial A} &  = 0 \qquad & & \mbox{on } \partial \Gamma(t), \quad t \in (0,T], \label{bc1} \\
w &= 0 \qquad & &\mbox{on } \partial \Gamma(t), \quad  t \in (0,T]. \label{bc2}
\end{alignat}
\end{subequations}
Here, $\nu_{\partial A}$ is the unit outward normal to $\partial A$, so that we assume that the evolving surfaces meet the boundary of the cylinder at a right angle. 
Finally, we impose  the initial conditions
\begin{equation}  \label{Gamma3}
\Gamma(0)=\Gamma^0, \qquad w(\cdot,0) = w^0  \mbox{ on } \Gamma^0,
\end{equation}
where $\Gamma^0=\lbrace (x,u^0(x)) \, | \, x \in \bar \Omega \rbrace$ and $u^0: \bar \Omega \rightarrow \mathbb R$ as well as  $w^0:\Gamma^0 \rightarrow \mathbb R$ are given functions. \\
The system (\ref{Gamma1}), (\ref{Gamma2}) occurs e.g. in the modeling of diffusion induced grain boundary motion, see \cite{FCE01}, \cite{DES01} and Section \ref{digm}. Further examples
of systems that arise by coupling a geometric evolution equation to a PDE on the evolving surface can be found in \cite[Section 10]{DzE13}. \\
A semi--discrete finite element scheme for the approximation of (\ref{Gamma1}), (\ref{Gamma2}) in the case that $\Gamma(t)$ is a closed curve has first been analysed by Pozzi and Stinner 
in \cite{PS17}. Using a tangentially modified parametrisation of the evolving curves, \cite{BDS17} obtains error bounds for a corresponding fully discrete scheme. In
\cite{SVY} this idea is applied to the case of open curves $\Gamma(t)$  meeting a given boundary orthogonally. In each of these papers the error bounds are optimal in $H^1$.
A first error analysis involving the evolution of two--dimensional closed (i.e. compact without boundary) surfaces was obtained in \cite{KLLP17} for a regularized version of (\ref{Gamma1}). Extending ideas used in the error analysis for pure mean curvature flow in \cite{KLL19}, Kov\'{a}cs, Li and Lubich obtain in \cite{KLL20} a convergence proof for the system (\ref{Gamma1}), (\ref{Gamma2}) in the case of closed surfaces.  The scheme  uses polynomials of degree at least two and is based on a system coupling  the variable $w$ in (\ref{Gamma2}) with 
the velocity, the normal and  the mean curvature of $\Gamma(t)$. The error estimates are optimal in $H^1$, while the restriction on the polynomial degree  
is essentially used to guarantee, via inverse estimates, that the discrete surfaces are non--degenerate. \\
The purpose of our paper is to derive and analyse a simple, fully discrete finite element scheme for the system (\ref{Gamma1}), (\ref{Gamma2}) when the evolving surfaces
are of the form (\ref{graph}). In order to translate \eqref{Gamma1}, \eqref{Gamma2} into problems which are posed on $\bar \Omega \times [0,T]$ we introduce
\begin{displaymath}
Q(u)= \sqrt{1+ | \nabla u |^2}.
\end{displaymath}
Then, the upward pointing unit normal $\nu(u)$, the normal velocity $V$ and the mean curvature $H$ of $\Gamma(t)$ are given by 
\begin{equation}  \label{VH}
\nu(u)= \frac{1}{Q(u)}(- \nabla u,1), \quad 
V= \frac{u_t}{Q(u)} \quad \mbox{ and } \quad H= \nabla \cdot \Bigl( \frac{ \nabla u}{Q(u)} \Bigr)
\end{equation}
respectively. 
Furthermore, if we denote by $n$ the outward unit normal to $\partial \Omega$, then
$\nu_{\partial A}=(n,0)$ and hence $\nu(u) \cdot \nu_{\partial A} = -  \frac{ \nabla u \cdot n}{Q(u)}$. 
If we let $\tilde w: \bar \Omega \times [0,T] \rightarrow \mathbb R, \, \tilde w(x,t):= w(x,u(x,t),t)$ then we may write \eqref{Gamma1}, \eqref{bc1} as
\begin{eqnarray} 
\frac{u_t}{Q(u)} -  \nabla \cdot \Bigl( \frac{ \nabla u}{Q(u)} \Bigr) + f(\tilde w) &=& 0 \quad  \mbox{ in } \Omega \times (0,T]; \label{strong} \\
\frac{ \nabla u \cdot n}{Q(u)} & = &  0 \quad \mbox{ on } \partial \Omega \times (0,T].  \label{strongbc}
\end{eqnarray}
Let us next rewrite \eqref{Gamma2} in terms of $\tilde w$. To do so, we make use of the formulae (2.1) and (2.2) in \cite{DzE13}, which yield (temporarily suppressing the dependence
on $t$)
\begin{align}
(\nabla_{\Gamma} w)(\Phi(x)) &= \sum_{i,j=1}^2 g^{ij}(x) \tilde w_{x_j}(x) \Phi_{x_i}(x), \label{tanggrad} \\
(\Delta_{\Gamma} w)(\Phi(x)) & =  \frac{1}{\sqrt{q(x)}} \sum_{i,j=1}^2 \frac{\partial }{\partial x_j} \Bigl( g^{ij}(x) \sqrt{q(x)} \tilde w_{x_i}(x) \Bigr). \label{label}
\end{align}
In the above, $\Phi(x)=(x,u(x))$ and $(g^{ij})$ is the inverse matrix of $(g_{ij})$, where $g_{ij}= \Phi_{x_i} \cdot \Phi_{x_j}= \delta_{ij} + u_{x_i} u_{x_j}, \, i,j=1,2$.
Furthermore, $q=\mbox{det} (g_{ij})= 1+ | \nabla u |^2=Q(u)^2$.
A simple calculation shows that
\begin{displaymath}
(g^{ij})= I- \frac{ \nabla u \otimes \nabla u}{Q(u)^2}.
\end{displaymath}
We can expand the velocity vector $(0,u_t)$ for the evolving family of graphs in terms of $\Phi_{x_1}, \Phi_{x_2}$ and $\nu(u)$ as follows
\begin{displaymath}
(0,u_t)= V \nu(u) + \sum_{k=1}^2 \frac{u_t u_{x_k}}{Q(u)^2} \Phi_{x_k}.
\end{displaymath}
Combining this relation with \eqref{tanggrad} we find
\begin{align*}
\tilde w_t &= w_t + \nabla w \cdot (0,u_t) = w_t +  V \frac{\partial w}{\partial \nu} + \nabla_{\Gamma} w \cdot \sum_{k=1}^2 \frac{u_t u_{x_k}}{Q(u)^2} \Phi_{x_k} \\
&= \partial^\bullet w + \sum_{i,j,k=1}^2 g^{ij} \tilde w_{x_j} \frac{u_t u_{x_k}}{Q(u)^2} \Phi_{x_i} \cdot \Phi_{x_k} 
 = \partial^\bullet w + \frac{u_t}{Q(u)^2} \nabla \tilde w \cdot \nabla u.
\end{align*}
Recalling (\ref{VH}) we deduce that
\begin{displaymath}
\partial^\bullet w - H \, V \,  w   = \tilde w_t - \frac{u_t}{Q(u)} \frac{ \nabla u}{Q(u)} \cdot \nabla \tilde w - \frac{u_t}{Q(u)} 
\nabla \cdot \Bigl( \frac{ \nabla u}{Q(u)} \Bigr) \, \tilde w 
 = \tilde w_t - \frac{u_t}{Q(u)} \nabla \cdot \Bigl( \tilde w \, \frac{ \nabla u}{Q(u)} \Bigr).
\end{displaymath}
Hence,  \eqref{Gamma2}, \eqref{bc2} take the form
\begin{eqnarray} 
\tilde w_t - \frac{1}{Q(u)} \sum_{i,j=1}^2 \Bigl( g^{ij} Q(u) \tilde w_{x_i} \Bigr)_{x_j} = \frac{u_t}{Q(u)} \nabla \cdot \Bigl( \tilde w \, \frac{ \nabla u}{Q(u)} \Bigr)
+ g(V, \tilde w) & & \mbox{ in } \Omega \times (0,T]; \label{strong1} \\
\tilde w  =  0 & &  \mbox{ on } \partial \Omega \times (0,T]. \label{strong1bc}
\end{eqnarray}
For ease of notation we will from now on write again $w$ instead of $\tilde w$. Our discretisation will be  based on a weak formulation
of the system (\ref{strong}), (\ref{strong1})  and uses continuous, piecewise linear finite elements in space and a backward Euler scheme in time, 
see Section \ref{sec:2}. A crucial point in the error analysis
is the  uniform control of the gradient of the discrete height function. This control is achieved with the help of a superconvergence estimate between the discrete height and a 
nonlinear projection previously employed  in \cite{DD00} for the numerical analysis of the mean curvature flow of graphs. The properties of this projection and a suitable 
projection for the function $w$ are collected in Section \ref{sec:3}. As our main results we obtain an $O(\tau+h)$--error bound in $H^1$ and an $O(\tau+ h^2 | \log h |^2)$--estimate
in $L^2$ both for $u$ and $w$, provided that the time step $\tau$  is appropriately related to the mesh size $h$. To  the best of our knowledge, a quasioptimal $L^2$--bound is new for 
coupled systems of the form (\ref{Gamma1}), (\ref{Gamma2}). The proof of the error bounds
is presented in Section \ref{sec:4} and split into two parts: for the analysis of the graph part we shall refer whenever possible to \cite{DD00} in order to keep the
presentation short. The analysis of the surface PDE requires much more work since the estimates have to be carried out in such a way as not to loose the optimal
order. Finally, in Section \ref{sec:5} we present several numerical tests that confirm our error estimates and apply the method in order to 
simulate diffusion induced grain boundary motion. \\
Let us finish the introduction with a few comments on our notation. We shall denote the norm of the Sobolev space 
$W^{m,p}(\Omega) \, (m \in \mathbb N_0, 1 \leq p \leq \infty)$ by $\Vert \cdot \Vert_{m,p}$. For $p=2$,
$W^{m,2}(\Omega)$ will be denoted by $H^m(\Omega)$ with norm $\Vert \cdot \Vert_m$, where we simply write $\Vert \cdot \Vert = \Vert \cdot \Vert_0$.

\section{Weak formulation and finite element approximation} \label{sec:2}
\setcounter{equation}{0}

\noindent
In what follows we make the following assumptions on the data and the solution $(u,w)$: \\
(A1) $f \in C^{0,1}_{\mbox{\footnotesize{loc}}}(\mathbb R)$ and $g:\mathbb R^2 \rightarrow \mathbb R$ has the form
\begin{equation} \label{formg}
g(r,s) = \alpha(r) \, \beta(s) + \tilde \beta(s),
\end{equation}
where $\beta, \tilde \beta \in C^{0,1}_{\mbox{\footnotesize{loc}}}(\mathbb R)$ and $\alpha(r)= \left\{ 
\begin{array}{ll}
\alpha_1 |r|, & r \geq 0, \\
\alpha_2 |r|, & r<0
\end{array}
\right.$ for some $\alpha_1,\alpha_2 \in \mathbb R$. \\[1.2mm]
(A2) $(u,w)$ solves  (\ref{strong}), (\ref{strongbc}), (\ref{strong1}), (\ref{strong1bc}) and satisfies
\begin{align}
& u \in L^\infty((0,T);H^4(\Omega)) \cap L^2((0,T);H^5(\Omega)), \, u_t \in L^\infty((0,T);H^2(\Omega)) \cap L^2((0,T);H^3(\Omega)) \label{regu1}   \\
& \nabla u_t \in L^\infty(\Omega \times (0,T)), \, u_{tt} \in L^\infty((0,T);H^1(\Omega)); \label{regu2} \\
& w \in C^0([0,T];W^{2,\infty}(\Omega)), \, w_t \in C^0([0,T]; W^{1,\infty}(\Omega) \cap H^2(\Omega)), w_{tt} \in L^\infty((0,T);L^2(\Omega)). \label{regw}
\end{align}
\noindent
Multiplying (\ref{strong}) by  $\varphi \in H^1(\Omega)$ and integrating by parts yields the weak formulation
\begin{equation} \label{Gamma1weak}
\int_{\Omega} \frac{u_t \, \varphi}{Q(u)} \,  dx + \int_{\Omega} \frac{\nabla u \cdot \nabla \varphi}{Q(u)} \,  dx = \int_{\Omega} f(w) \, \varphi \, dx 
\qquad \forall \varphi \in H^1(\Omega).
\end{equation}
In order to derive a weak formulation for (\ref{strong1}) we proceed as in \cite[Section 5]{DzE13} and calculate for a test function $\eta \in H^1_0(\Omega)$
\begin{eqnarray}
\lefteqn{ \hspace{-1.5cm}
\frac{d}{dt} \int_{\Omega}  w \, \eta \, Q(u) dx = \int_{\Omega} w_t \, \eta \, Q(u) dx + \int_{\Omega}  w \, \eta \, [Q(u)]_t dx =
 \sum_{i,j=1}^2 \int_{\Omega} \Bigl( g^{ij} w_{x_i} \, Q(u) \Bigr)_{x_j}  \, \eta \, dx } \nonumber  \\
 & &  \quad + \int_{\Omega} u_t \nabla \cdot \Bigl(  w \, \frac{ \nabla u}{Q(u)} \Bigr) \, \eta \,  dx
+ \int_{\Omega}  w \, \eta \, \frac{ \nabla u \cdot \nabla u_t}{Q(u)} dx + \int_{\Omega} g(V, w) \eta \, Q(u) \, dx \nonumber  \\
&=& - \sum_{i,j=1}^2  \int_{\Omega} g^{ij} w_{x_i} \, \eta_{x_j} Q(u) dx - \int_{\Omega} u_t \frac{\nabla u \cdot \nabla \eta}{Q(u)} \, w dx +  \int_{\Omega} 
g(V, w) \, \eta \, Q(u) \, dx  \nonumber  \\
& = &  - \int_{\Omega} E(\nabla u) \nabla w \cdot \nabla \eta dx - \int_{\Omega} \nabla u \cdot \nabla \eta \, V w dx +  \int_{\Omega} 
g(V, w) \, \eta \, Q(u) \, dx, \label{Gamma2weak}
\end{eqnarray}
where $V$ is given by (\ref{VH}) and
\begin{equation} \label{defE}
E(p)= \sqrt{1+ |p|^2} \bigl( I- \frac{p \otimes p}{1+ |p|^2} \bigr), \quad p \in \mathbb R^2.
\end{equation}
Note that for all $p, \xi \in \mathbb R^2$
\begin{equation} \label{ellipt}
E(p) \xi \cdot \xi = \sqrt{1+ |p|^2} \bigl( |\xi |^2 - \frac{(\xi \cdot p)^2}{1+ |p|^2} \bigr) \geq \sqrt{1+|p|^2} \,  | \xi|^2 \bigl( 1- \frac{| p|^2}{1+ |p|^2} \bigr)
= \frac{| \xi|^2}{\sqrt{1+ |p|^2}}.
\end{equation}

\noindent
Next, let $(\mathcal T_h)_{0< h \leq h_0}$ be a family of triangulations of $\Omega$, where we allow boundary elements to have
one curved face in order to avoid the analysis of domain approximation. We denote by $h:=\max_{S \in \mathcal T_h} \mbox{diam}(S)$ the
maximum mesh size and assume that the triangulation is quasiuniform in the sense that there exists $\kappa>0$ which is independent of $h$,
such that each $S \in \mathcal T_h$ is contained in a ball of radius $\kappa^{-1} h$ and contains a ball of radius $\kappa h$. Our finite
element spaces are given by
\begin{displaymath}
X_h = \lbrace \varphi_h \in C^0(\bar \Omega) \, | \, \varphi_h \mbox{ is a linear polynomial on each } S \in \mathcal T_h \rbrace, \quad X_{h0}= X_h \cap H^1_0(\Omega)
\end{displaymath}
with an appropriate modification in the curved elements. We refer to \cite{Z73} for a construction of $X_h$. The following well--known estimates will be useful:
\begin{eqnarray}
 \Vert \nabla \varphi_h \Vert & \leq &  c h^{-1} \Vert \varphi_h \Vert \qquad \quad \quad  \forall \varphi_h \in X_h; \label{invers} \\
 \Vert \nabla \varphi_h \Vert_{0,\infty} & \leq & c h^{-1} \Vert \nabla \varphi_h \Vert \qquad \; \quad \forall \varphi_h \in X_h; \label{invers1} \\
\Vert \varphi_h \Vert_{0,\infty} & \leq & c | \log h|^{\frac{1}{2}} \Vert \varphi_h \Vert_1 \qquad \forall \varphi_h \in X_h.\label{embed}
\end{eqnarray}
Finally, let $\tau>0$ be a time step
and $t_m=m \tau, ~m=0,\ldots,M$, where $M=\frac{T}{\tau}$. In what follows, an upper index $m$ will refer to the time level $m$.  \\[2mm]
Our discretisation reads: Given $u^m_h \in X_h, \, w^m_h \in X_{h0}$,  first find $u^{m+1}_h \in X_h$ such that
\begin{align} 
& \frac{1}{\tau} \int_{\Omega} \frac{(u^{m+1}_h - u^m_h) \, \varphi_h}{Q(u^m_h)} \,  dx + \int_{\Omega} \frac{\nabla u^{m+1}_h \cdot \nabla \varphi_h}{Q(u^m_h)} \,  dx
= \int_{\Omega} f(w^m_h) \, \varphi_h \, dx  \label{disc1} 
\end{align}
for all $\varphi_h \in X_h$. Afterwards, find $w^{m+1}_h \in X_{h0}$ such that
\begin{align}
& \frac{1}{\tau} \Bigl( \int_{\Omega} w^{m+1}_h \, \eta_h \, Q(u^{m+1}_h) dx - \int_{\Omega} w^m_h \, \eta_h Q(u^m_h) dx \Bigr)  +  \int_{\Omega} 
E(\nabla u^{m+1}_h) \nabla w^{m+1}_h \cdot \nabla \eta_h dx \nonumber \\
& \qquad = - \int_{\Omega}  \nabla u^{m+1}_h \cdot \nabla \eta_h \, V^{m+1}_h \,  w^m_h dx +  \int_{\Omega} g(V^{m+1}_h, w^m_h) \, \eta_h \, Q(u_h^{m+1}) \,  dx \label{disc2} 
\end{align}
for all $\eta_h \in X_{h0}$. Here, $V^{m+1}_h=\frac{1}{\tau}\frac{u^{m+1}_h-u^m_h}{Q(u^{m+1}_h)}$. 
We note that each time step requires the consecutive solution of two
linear systems. In view of (\ref{ellipt}) it is easily seen that $u^{m+1}_h \in X_h$
and $w^{m+1}_h \in X_{h0}$ exist and are uniquely determined.
The algorithm is initialised by $u^0_h=\hat u^0_h, w^0_h = \hat w^0_h$, given by (\ref{msproj}) and (\ref{wproj}) defined in the next section. 
\noindent
Our main result reads as follows:
\begin{thm} \label{main} There exist $h_0>0, ~\delta_0>0$ such that for all $0<h \leq h_0$ and all $\tau>0$ satisfying $\tau  \leq \delta_0 h | \log h |^{-\frac{1}{2}}$ the following error bounds hold:
\begin{eqnarray*}
\max_{0 \leq m \leq M} \left[ \Vert u^m - u^m_h \Vert^2 +  \Vert w^m - w^m_h \Vert^2 \right] + \sum_{m=0}^{M-1} \tau 
\Vert u^m_t - \frac{u^{m+1}_h -u^m_h}{\tau} \Vert^2 & \leq & c \bigl( \tau^2 + h^4 | \log h |^4 \bigr), \\
\max_{0 \leq m \leq M} \Vert \nabla ( u^m - u^m_h) \Vert^2 + \sum_{m=0}^M \tau \Vert \nabla (w^m - w^m_h) \Vert^2 & \leq & c \bigl( \tau^2 + h^2 \bigr). 
\end{eqnarray*}
\end{thm}

\section{Projections} \label{sec:3}
\setcounter{equation}{0}

Our error analysis relies on the use of suitable  Ritz projections of the solutions $u$ and $w$. Omitting the time dependence for a moment we
define for a given function $u \in  H^1(\Omega)$ the minimal surface type projection $\hat u_h \in X_h$ by
\begin{equation} \label{msproj}
\int_{\Omega} \frac{\nabla \hat u_h \cdot \nabla \varphi_h}{Q(\hat u_h)} \,  dx + \int_{\Omega} \hat u_h \, \varphi_h \, dx =
\int_{\Omega} \frac{\nabla u \cdot \nabla \varphi_h}{Q(u)} \,  dx + \int_{\Omega} u \, \varphi_h \, dx \qquad \forall \varphi_h \in X_h.
\end{equation}
Note that we have added the zero order term in order to ensure the $H^1(\Omega)$--coercivity of the problem. For functions that
also depend on $t$ we have the following error bounds.
\begin{lem} \label{proju} Assume that $u$ satisfies (\ref{regu1}) and (\ref{regu2}). Then
\begin{eqnarray} 
\sup_{0 \leq t \leq T}\Vert (u - \hat u_h)(t) \Vert + h \sup_{0 \leq t \leq T} \Vert \nabla (u - \hat u_h)(t) \Vert & \leq & c h^2, \label{ms1} \\
\sup_{0 \leq t \leq T} \Vert (u - \hat u_h)(t) \Vert_{0,\infty} + h \sup_{0 \leq t \leq T} \Vert \nabla (u - \hat u_h)(t) \Vert_{0,\infty} & \leq & c h^2 | \log h |, \label{ms2} \\
\sup_{0 \leq t \leq T} \Vert (u_t - \hat u_{h,t})(t) \Vert & \leq & c h^2 | \log h|^2, \label{ms3} \\
\sup_{0 \leq t \leq T} \Vert \nabla (u_t - \hat u_{h,t})(t) \Vert & \leq & c h. \label{ms4}
\end{eqnarray}
\end{lem}
\begin{proof} The proofs of (\ref{ms1}) and (\ref{ms2}) follow from \cite{DR80} (see p. 160) using that $u(\cdot,t) \in H^4(\Omega) \subset W^{2,\infty}(\Omega)$ for
every $t \in [0,T]$. The arguments required to show   (\ref{ms3}) and (\ref{ms4}) can be found in 
\cite[Section 4]{DD95} for the case of homogeneous Dirichlet boundary conditions. In order to prove (\ref{ms4}) for  (\ref{msproj}) one proceeds in the same way as in 
\cite{DD95}, p. 202 to obtain
\begin{displaymath}
\Vert \nabla (u_t - \hat u_{h,t}) \Vert^2 \leq c h  \Vert \nabla (u_t - \hat u_{h,t}) \Vert \bigl( \Vert \nabla u_t \Vert_{0,\infty} + \Vert u_t \Vert_2 \bigr) + c h^2
\Vert \nabla u_t \Vert_{0,\infty} \Vert u_t \Vert_2,
\end{displaymath}
which yields (\ref{ms4}) taking into account (\ref{regu1}) and (\ref{regu2}). The bound (\ref{ms3})  can be shown  for the Neumann case by modifying the 
dual problem on top of p. 203 in \cite{DD95} as follows:
\begin{displaymath}
- \nabla \cdot \bigl( F'(\nabla u) \nabla v \bigr) + v = u_t - \hat u_{h,t} \quad \mbox{ in } \Omega, \quad F'(\nabla u) \nabla v \cdot n = 0 \quad 
\mbox{ on } \partial \Omega,
\end{displaymath}
where $F(p)=p/\sqrt{1+|p|^2}, \, p \in \mathbb R^2$.
\end{proof}

\noindent
Let us next use $\hat u_h$ in order to define a projection $\hat w_h \in X_{h0}$ of $w$ as follows:
\begin{equation} \label{wproj}
\int_{\Omega} E(\nabla \hat u_h) \nabla \hat w_h \cdot \nabla \eta_h \, dx = \int_{\Omega} E(\nabla u) \nabla w \cdot \nabla \eta_h \, dx \quad \quad 
\forall \eta_h \in X_{h0}.
\end{equation}
\begin{lem} \label{projw} Assume that $w$ satisfies (\ref{regw}). Then
\begin{eqnarray}
\sup_{0 \leq t \leq T} \Vert \nabla( w - \hat w_h)(t) \Vert & \leq &  c h, \label{w1} \\
\sup_{0 \leq t \leq T} \Vert (w - \hat w_h)(t) \Vert & \leq & c h^2 | \log h |, \label{w2} \\
\sup_{0 \leq t \leq T} \Vert \nabla (w_t - \hat w_{h,t})(t) \Vert & \leq & c h, \label{w3} \\
\sup_{0 \leq t \leq T} \Vert (w_t - \hat w_{h,t})(t) \Vert & \leq & c h^2 | \log h|^2. \label{w4}
\end{eqnarray}
\end{lem}
\begin{proof} Using (\ref{ms1})--(\ref{ms4}), these bounds have been obtained in \cite[Appendix]{DD06} for a slightly more complicated projection, see (2.22) in that paper.
The same arguments can be applied to our case where we note that the matrix valued function $E(p)$ used in \cite{DD06} differs from (\ref{defE}) by
a factor of $1+|p|^2$. However, since $\nabla u$ and $\nabla \hat u_h$ vary in a bounded set that is independent of $h$, the analysis in \cite{DD06}
also applies to (\ref{wproj}).
\end{proof}

\noindent
For later use we record the following estimates, which will be helpful in retaining the optimality of the error bounds:
\begin{lem} \label{Fest}  Suppose that $F: \mathbb R^2 \rightarrow \mathbb R$ is twice continuously differentiable and that $u \in W^{2,\infty}(\Omega)$. Then
we have for $f \in W^{1,1}_0(\Omega)$
\begin{displaymath}
| \int_{\Omega} \bigl( F(\nabla u) - F(\nabla \hat u_h) \bigr) \, f \, dx | \leq c h^2 | \log h | \, \Vert f \Vert_{1,1}.
\end{displaymath}
\end{lem}
\begin{proof} Abbreviating $\rho_u=u-\hat u_h$ we have
\begin{displaymath}
 \int_{\Omega} \bigl( F(\nabla u) - F(\nabla \hat u_h) \bigr) \, f \, dx = \int_{\Omega} F'(\nabla u) \cdot \nabla \rho_u \, f \, dx + R,
\end{displaymath}
where
\begin{displaymath}
| R | = \big| \int_{\Omega} \int_0^1 \bigl( F'(\nabla u - s \nabla \rho_u) - F'(\nabla u) \bigr) ds \cdot \nabla \rho_u \, f \, dx \big| \leq c \Vert \nabla \rho_u \Vert_{0,\infty}
\Vert \nabla \rho_u \Vert \Vert f \Vert \leq c h^2 | \log h | \Vert f \Vert_{1,1}
\end{displaymath}
in view of (\ref{ms1}) and (\ref{ms2}) and the embedding $W^{1,1}(\Omega) \hookrightarrow L^2(\Omega)$. Integration by parts together with (\ref{ms2})
yields
\begin{displaymath}
 \big| \int_{\Omega} F'(\nabla u) \cdot \nabla \rho_u \, f \, dx  \big|  =    \big| - \int_{\Omega} \nabla \cdot \bigl( F'(\nabla u) \, f \bigr) \rho_u \, dx  \big|
 \leq  c \Vert \rho_u \Vert_{0,\infty} \Vert f \Vert_{1,1} \leq c h^2 | \log h |  \Vert f \Vert_{1,1}
\end{displaymath}
and the result follows.
\end{proof}

\begin{lem} \label{Fest1} Suppose that $f \in H^1_0(\Omega)\cap C^0(\bar \Omega)$ with $f \in H^2(T)$ for all $T \in \mathcal T_h$. Then
\begin{displaymath}
 \big| \int_{\Omega} f  \bigl( \frac{\nabla u}{Q(u)} - \frac{\nabla \hat u_h}{Q(\hat u_h)} \bigr) \cdot \nabla \varphi_h \, dx  \big| \leq  
c h | \log h | \Vert \varphi_h \Vert   \Bigl( \sum_{T \in \mathcal T_h} \Vert f \Vert_{H^2(T)}^2  \Bigr)^{\frac{1}{2}} \qquad \forall \varphi_h \in X_h.
\end{displaymath}
If in addition, $f \in H^2(\Omega)$, then 
\begin{displaymath}
 \big| \int_{\Omega} f \bigl( \frac{\nabla u}{Q(u)} - \frac{\nabla \hat u_h}{Q(\hat u_h)} \bigr) \cdot \nabla \varphi_h \, dx  \big| \leq  
c h^2 | \log h | \Vert \varphi_h \Vert_1 \Vert f \Vert_2 \qquad \forall \varphi_h \in X_h.
\end{displaymath} 
\end{lem}
\begin{proof} In view of the definition (\ref{msproj}) of $\hat u_h$ we obtain
\begin{eqnarray*}
\lefteqn{
\int_{\Omega} f \bigl( \frac{\nabla u}{Q(u)} - \frac{\nabla \hat u_h}{Q(\hat u_h)}  \bigr) \cdot \nabla \varphi_h \, dx  } \\
& = & \int_{\Omega}  \bigl( \frac{\nabla u}{Q(u)} - \frac{\nabla \hat u_h}{Q(\hat u_h)} \bigr) \cdot \nabla ( f \,  \varphi_h ) \, dx -
\int_{\Omega} \varphi_h  \bigl( \frac{\nabla u}{Q(u)} - \frac{\nabla \hat u_h}{Q(\hat u_h)} \bigr)  \cdot \nabla f \, dx \\
& = & \int_{\Omega}  \bigl( \frac{\nabla u}{Q(u)} - \frac{\nabla \hat u_h}{Q(\hat u_h)} \bigr) \cdot \nabla \bigl( f \varphi_h - I_h(f \varphi_h) \bigr) \, dx
+ \int_{\Omega} \rho_u I_h( f \varphi_h) \, dx \\
& &   -\int_{\Omega} \varphi_h  \bigl( \frac{\nabla u}{Q(u)} - \frac{\nabla \hat u_h}{Q(\hat u_h)} \bigr)  \cdot \nabla f \, dx
=: I+ II +III.
\end{eqnarray*}
Here, $I_h$ denotes the Lagrange interpolation operator. An interpolation estimate implies
\begin{eqnarray*}
| I |  & \leq & C \Vert \nabla \rho_u \Vert_{0,\infty} \Vert\nabla \bigl(  f \varphi_h - I_h( f \varphi_h)   \bigr)\Vert_{0,1} \leq 
c h^2 | \log h | \sum_{T \in \mathcal T_h} \Vert D^2( f \varphi_h) \Vert_{L^1(T)} \\
& \leq & c h^2 | \log h | \Vert \varphi_h \Vert_1 \bigl(\sum_{T \in \mathcal T_h} \Vert f \Vert_{H^2(T)}^2  \bigr)^{\frac{1}{2}}.
\end{eqnarray*}
Next,
\begin{eqnarray*}
| II | & \leq & \Vert \rho_u \Vert_{0,\infty}  \Vert I_h ( f \varphi_h) \Vert_{0,1} \leq c h^2 | \log h | \bigl( \Vert f \varphi_h \Vert_{0,1} + \Vert f \varphi_h - I_h( f \varphi_h) 
\Vert_{0,1} \bigr) \\
& \leq & c h^2 | \log h | \Vert \varphi_h \Vert_1 \bigl(\sum_{T \in \mathcal T_h} \Vert f \Vert_{H^2(T)}^2  \bigr)^{\frac{1}{2}}.
\end{eqnarray*}
Finally,
\begin{displaymath}
| III | \leq  C \Vert \nabla \rho_u \Vert_{0,\infty} \Vert \varphi_h \Vert \, \Vert f \Vert_1 \leq c h | \log h | \Vert \varphi_h \Vert \, \Vert f \Vert_1, 
\end{displaymath}
while Lemma \ref{Fest} yields in the case that $f \in H^2(\Omega)$
\begin{displaymath}
| III | \leq c h^2 | \log h | \Vert  \varphi_h \nabla f  \Vert_{1,1} \leq c h^2  | \log h | \Vert \varphi_h \Vert_1 \Vert f \Vert_2.
\end{displaymath}
The result now follows from the above bounds together with (\ref{invers}).
\end{proof}

\section{Error Analysis} \label{sec:4}
\setcounter{equation}{0}

\noindent
Let us begin with two useful estimates involving the quantities $Q$ and $\nu$.

\begin{lem} Let $u,v \in W^{1,\infty}(\Omega)$. Then we have a.e. in $\Omega$:
\begin{eqnarray}
| \nabla (v-u) | & \leq & \bigl( 1+ \sup_\Omega | \nabla v| \bigr) Q(u) | \nu(v) - \nu(u) |; \label{q1} \\
Q(v)-Q(u) & = & \frac{\nabla u}{Q(u)} \cdot \nabla (v-u) + \frac{| \nabla (v-u) |^2}{2 Q(u)} - \frac{(Q(v)-Q(u))^2}{2 Q(u)}. \label{q2}
\end{eqnarray}
\end{lem}
\begin{proof} The estimate  (\ref{q1}) is a consequence of the relation
\begin{displaymath}
\nabla v - \nabla u = Q(u) \bigl( \frac{\nabla v}{Q(v)} - \frac{\nabla u }{Q(u)} \bigr) + Q(u) \bigl( \frac{1}{Q(u)} - \frac{1}{Q(v)} \bigr) \nabla v
\end{displaymath}
and the fact that $\nu(u)=\bigl(\frac{-\nabla u}{Q(u)}, \frac{1}{Q(u)} \bigr)$, while (\ref{q2}) follows from a straightforward calculation.
\end{proof}

\noindent
Let us decompose the errors $e_u^m=u^m- u^m_h, \, e_w^m=w^m-w^m_h$ as follows:
\begin{eqnarray}
e_u^m &=& (u^m- \hat u^m_h) + (\hat u^m_h - u^m_h)=: \rho^m_u+ e^m_{h,u}, \label{defeu} \\
e_w^m &=& (w^m- \hat w^m_h) + (\hat w^m_h - w^m_h)=: \rho^m_w+ e^m_{h,w} \label{defew}
\end{eqnarray}
and note that $e^m_{h,u} \in X_h, \, e^m_{h,w} \in X_{h0}$. It will be convenient to introduce the quantities
\begin{eqnarray} 
E^m &:= &  \int_{\Omega} | \nu(u^m_h) - \nu(\hat u^m_h) |^2 Q(u^m_h) \, dx, \label{defEm} \\
F^m &:= & \frac{1}{2} E^m - \int_{\Omega} d^m \cdot \nabla e^m_{h,u} \, \rho^m_u \, dx, \label{defF}
\end{eqnarray}
where
\begin{equation} \label{defdm}
d^m= - \frac{ u^m_t \nabla u^m}{\sqrt{1+ | \nabla u^m |}^3}.
\end{equation}
\noindent
We shall use an induction argument and claim that
\begin{equation} \label{induction}
F^m + \frac{\beta^2}{2} \int_{\Omega} (e^m_{h,w})^2 Q(u^m_h) \, dx \leq \bigl( \tau^2+h^4 | \log h |^4 \bigr) \, e^{\mu t_m}, \quad
m=0,1,\ldots,M
\end{equation}
provided that $\tau \leq \delta_0  h  | \log h |^{- \frac{1}{2}}$. 
The constants $\delta_0, ~0< \beta \leq 1$ and $\mu>0$ are independent of $h$ and $\tau$ and will be chosen a posteriori. To begin, choose $h_0>0$ so small that
\begin{equation} \label{defh0}
 h^2 | \log h  |^5 e^{\mu T} \leq \frac{1}{2} \mbox{ and } | \log h | \geq \frac{1}{\beta^2} \qquad \mbox{ for all } 0< h \leq h_0.
\end{equation}
Clearly, (\ref{induction}) holds for $m=0$ since $e^0_{h,u}=e^0_{h,w}=0$ by the choice
of our initial data for the scheme. 
Let us assume that it is true for some $m \in \lbrace 0,\ldots,M-1 \rbrace$. Then we have for $0<h \leq h_0$ that
\begin{equation} \label{embound}
F^m + \frac{\beta^2}{2} \int_{\Omega} (e^m_{h,w})^2 Q(u^m_h) \, dx  \leq \bigl( \delta_0^2 h^2 | \log h |^{-1} + h^4 | \log h |^4 \bigr) e^{\mu T} \leq
h^2 | \log h |^{-1},
\end{equation}
provided that $\delta_0$ and  $\mu$ satisfy
\begin{equation} \label{cond1}
\delta_0^2 e^{\mu T} \leq \frac{1}{2}.
\end{equation}
In what follows we shall denote by $c$ a generic constant that is independent of $\delta_0,\beta$ and $\mu$.
We infer from an inverse estimate, (\ref{embound}), the fact that $Q(u^m_h) \geq 1$ and (\ref{defh0}) that
\begin{equation} \label{whbound}
\Vert w^m_h \Vert_{0,\infty} \leq  \Vert \hat w^m_h \Vert_{0,\infty} + \Vert e^m_{h,w} \Vert_{0,\infty} \leq c + c h^{-1} \Vert e^m_{h,w} \Vert \leq 
c+ \frac{c}{\beta} | \log h |^{-\frac{1}{2}} \leq c. 
\end{equation}
\noindent
Next, we deduce with the help of $\Vert \nabla \hat u^m_h \Vert_{0,\infty} \leq c$ and  (\ref{invers1}) that  
\begin{equation}  \label{qest}
 \sup_{\bar \Omega} Q(u^m_h) \leq 1 + \sup_{\bar \Omega} | \nabla u^m_h | \leq 1 + \Vert \nabla \hat u^m_h \Vert_{0,\infty} + \Vert \nabla e^m_{h,u} \Vert_{0,\infty} 
 \leq c + c h^{-1} \Vert \nabla e^m_{h,u} \Vert.
\end{equation}
It follows from (\ref{q1})   that 
\begin{align*}
| \nabla e^m_{h,u} | = | \nabla (u^m_h - \hat u^m_h) | \leq (1+ \sup_{\bar \Omega} | \nabla \hat u^m_h | ) Q(u^m_h) | \nu(u^m_h) - \nu(\hat u^m_h) | \leq
c | \nu(u^m_h) - \nu(\hat u^m_h) | Q(u^m_h).
\end{align*}
Thus,
\begin{align*}
 \Vert \nabla e^m_{h,u} \Vert^2 & \leq c \int_{\Omega} | \nu(u^m_h) - \nu(\hat u^m_h) |^2 Q(u^m_h)^2 \, dx \leq c \sup_{\bar \Omega} Q(u^m_h) E^m \\
& \leq c \sup_{\bar \Omega} Q(u^m_h) \bigl( F^m + \Vert \rho^m_u \Vert  \Vert \nabla e^m_{h,u} \Vert \bigr) \leq
c \sup_{\bar \Omega} Q(u^m_h) \bigl( F^m + h^2  \Vert \nabla e^m_{h,u} \Vert \bigr)
\end{align*}
and hence
\begin{equation} \label{emhuest}
\Vert \nabla e^m_{h,u} \Vert^2 \leq c \sup_{\bar \Omega} Q(u^m_h) F^m + c h^4 (\sup_{\bar \Omega} Q(u^m_h))^2.
\end{equation}
If we insert this bound into (\ref{qest}) and recall (\ref{embound})  we obtain
\begin{displaymath}
\sup_{\bar \Omega} Q(u^m_h) \leq c + c \bigl( | \log h |^{-1} \,  \sup_{\bar \Omega} Q(u^m_h) \bigr)^{\frac{1}{2}} + c h \sup_{\bar \Omega} Q(u^m_h)
\end{displaymath}
and therefore
\begin{equation} \label{qbound}
\sup_{\bar \Omega} Q(u^m_h) \leq c
\end{equation}
provided that $0< h \leq h_1$ for some sufficiently small $0<h_1 \leq h_0$. Furthermore, we infer from (\ref{embound}), (\ref{emhuest}) and (\ref{qbound})  that
\begin{eqnarray} 
\Vert \nabla e^m_{h,u} \Vert^2 & \leq & c F^m + c h^4 \leq c h^2 | \log h |^{-1}, \label{embound2} \\
\frac{1}{2} E^m & \leq & F^m + c \Vert \nabla e^m_{h,u} \Vert \Vert \rho^m_u \Vert \leq ch^2 | \log h |^{-1}. \label{embound1} 
\end{eqnarray}

\subsection{The graph equation} 
\noindent
Evaluating (\ref{Gamma1weak}) at $t=t_m$ and using the definition (\ref{msproj}) of $\hat u_h$ we derive for $\varphi_h \in X_h$
\begin{displaymath}
\int_{\Omega} \frac{u^m_t \, \varphi_h}{Q(u^m)} \,  dx + \int_{\Omega} \frac{\nabla \hat u^m_h \cdot \nabla \varphi_h}{Q(\hat u^m_h)} \,  dx =
\int_{\Omega} f( w^m) \, \varphi_h \, dx  + \int_{\Omega} \rho^m_u \, \varphi_h \, dx
\end{displaymath}
and hence
\begin{align}
& \frac{1}{\tau} \int_{\Omega} \frac{(u^{m+1}-u^m) \varphi_h}{Q(u^m)} \, dx + \int_{\Omega} \frac{\nabla \hat u^{m+1}_h \cdot \nabla \varphi_h}{Q(\hat u^m_h)} \,  dx  \label{erroru}  \\
& \qquad  = \int_{\Omega} f( w^m) \, \varphi_h \, dx 
 + \int_{\Omega} \frac{\nabla( \hat u^{m+1}_h - \hat u^m_h) \cdot \nabla \varphi_h}{Q(\hat u^m_h)} \, dx +  \int_{\Omega} R^m \, \varphi_h \, dx. \nonumber
\end{align}
Here, $R^m= \frac{1}{Q(u^m)} \bigl( \frac{u^{m+1}-u^m}{\tau} - u^m_t \bigr) + \rho^m_u$, so that in view of (\ref{ms1})
\begin{equation} \label{rmest}
\Vert R^m \Vert \leq \int_{t_m}^{t_{m+1}} \Vert u_{tt} \Vert \, dt + \Vert \rho^m_u\Vert \leq c \bigl( \tau+ h^2).
\end{equation}
Combining (\ref{erroru}) with (\ref{disc1}) we obtain the error relation
\begin{align}
& \frac{1}{\tau} \int_{\Omega} \frac{(e^{m+1}_u-e^m_u) \varphi_h }{Q(u^m_h)} \, dx + \int_{\Omega} \Bigl( \frac{\nabla \hat u^{m+1}_h}{Q(\hat u^m_h)} - \frac{\nabla u^{m+1}_h}{Q(u^m_h)}
\Bigr) \cdot \nabla \varphi_h \, dx = \int_{\Omega} \bigl( f(w^m) - f(w^m_h) \bigr) \varphi_h \, dx \nonumber \\
& \quad +  \frac{1}{\tau} \int_{\Omega}(u^{m+1}-u^m) \Bigl( \frac{1}{Q(u^m_h)} - \frac{1}{Q(u^m)} \Bigr) \, \varphi_h \, dx 
+ \int_{\Omega} R^m \varphi_h \, dx +
\int_{\Omega} \frac{\nabla( \hat u^{m+1}_h - \hat u^m_h) \cdot \nabla \varphi_h}{Q(\hat u^m_h)} \, dx. \label{erroru1}
\end{align}
If we insert $\varphi_h= \frac{1}{\tau} \bigl( e^{m+1}_{h,u} - e^m_{h,u} \bigr)$ into (\ref{erroru1})  we derive
\begin{align}
&  \frac{1}{\tau^2} \int_{\Omega} \frac{( e^{m+1}_u - e^m_u)^2}{Q(u^m_h)} \, dx + \frac{1}{\tau}  \int_{\Omega} \Bigl( \frac{\nabla \hat u^{m+1}_h}{Q(\hat u^m_h)} - \frac{\nabla u^{m+1}_h}{Q(u^m_h)} \Bigr) \cdot \nabla (e^{m+1}_{h,u} - e^m_{h,u}) \, dx  \nonumber   \\
& =  \frac{1}{\tau^2} \int_{\Omega} \frac{(e^{m+1}_u - e^m_u)(\rho^{m+1}_u - \rho^m_u)}{Q(u^m_h)} \, dx 
+ \frac{1}{\tau} \int_{\Omega} \frac{\nabla( \hat u^{m+1}_h - \hat u^m_h) \cdot \nabla (e^{m+1}_{h,u} - e^m_{h,u})}{Q(\hat u^m_h)} \, dx \nonumber  \\
& \quad +   \frac{1}{\tau^2} \int_{\Omega}(u^{m+1}-u^m)  \, (e^{m+1}_{h,u} - e^m_{h,u})  \, \Bigl( \frac{1}{Q(u^m_h)} - \frac{1}{Q(u^m)} \Bigr) dx
+ \frac{1}{\tau}  \int_{\Omega} R^m (e^{m+1}_{h,u} - e^m_{h,u}) \, dx \nonumber \\
& \quad + \frac{1}{\tau}  \int_{\Omega} \bigl( f(w^m) - f(w^m_h) \bigr)
(e^{m+1}_{h,u} - e^m_{h,u}) \, dx \quad   =: \sum_{i=1}^5 A_i. \label{err2}
\end{align}
In order to proceed we make use of the analysis in \cite{DD00} for the mean curvature flow of graphs subject to Dirichlet boundary conditions. The relation (\ref{err2})
corresponds to \cite[(3.12)]{DD00} where we use $e^m_u,e^m_{h,u}, \rho^m_u$ instead of $e^m,e^m_h,\epsilon^m$ respectively. Furthermore, our remainder term $R^m$ is defined in a different way and the term $A_5$ is not present in \cite{DD00}. We shall refer to the calculations in \cite{DD00} whenever possible and focus on the changes due to
the differences mentioned above and the use of a Neumann boundary condition. To begin, it follows from Lemma 2 in \cite{DD00} that
\begin{align}
& \frac{1}{\tau} \int_{\Omega} \Bigl( \frac{\nabla \hat u^{m+1}_h}{Q(\hat u^m_h)} - \frac{\nabla u^{m+1}_h}{Q(u^m_h)} \Bigr)
\cdot \nabla (e^{m+1}_{h,u} - e^m_{h,u}) \, dx \geq \frac{1}{2 \tau } ( E^{m+1} - E^m) \nonumber \\
& \qquad + \frac{1}{4 \tau} \int_{\Omega} \frac{| \nabla (e^{m+1}_{h,u} - e^m_{h,u})|^2}{Q(\hat u^m_h)} \, dx - c (E^m+ E^{m+1}) - c \tau^2. \label{err3}
\end{align} 
The lemma holds under the condition that $h^{-2} E^m \leq \gamma$ and $\gamma>0$ is sufficiently small, which can be achieved in view of (\ref{embound1}) if $0<h \leq h_2$ and
$h_2 \leq h_1$ is small enough. \\
Let us consider the terms on the right hand side of (\ref{err2}). The term $S_1$ is estimated in (i) at the bottom of page 352 in \cite{DD00}, so that
\begin{align}
| A_1 | \leq \frac{\delta}{\tau^2} \int_{\Omega} \frac{( e^{m+1}_u - e^m_u)^2}{Q(u^m_h)} \, dx + \frac{c}{\delta} h^4 | \log h|^4. \label{err4}
\end{align}
The integral $A_2$ is treated in (ii) on page 353 in \cite{DD00} and uses integration by parts for the term
\begin{displaymath}
\int_{\Omega} \frac{\nabla (u^{m+1}-u^m) \cdot \nabla (e^{m+1}_{h,u} - e^m_{h,u})}{Q(u^m)} \, dx.
\end{displaymath}
Since $\nabla (u^{m+1} - u^m) \cdot n=0$ on $\partial \Omega$ in view of (\ref{strongbc}) the boundary integral vanishes and we obtain in the same way as
in \cite{DD00}
\begin{align} \label{err5}
| A_2 | \leq \frac{\delta}{\tau} \int_{\Omega} \frac{| \nabla (e^{m+1}_{h,u} - e^m_{h,u})|^2}{Q(\hat u^m_h)} \, dx+ 
\frac{\delta}{\tau^2} \int_{\Omega} \frac{( e^{m+1}_u - e^m_u)^2}{Q(u^m_h)} \, dx + \frac{c}{\delta} \bigl( \tau^2 + h^4 | \log h |^4 \bigr).
\end{align}
The term $A_3$ is handled in (iii) on pages 353 to 356 in \cite{DD00}. It again involves an integration by parts, namely for the term
\begin{displaymath}
- \frac{1}{\tau^2} \int_{\Omega} (u^{m+1}- u^m) (e^{m+1}_{h,u} - e^m_{h,u}) b^m \cdot  \nabla \rho^m_u \, dx,
\end{displaymath}
which is $II$ at the top of page 354. Here $b^m=B(\nabla u^m)$ with $B_i(p)= \frac{\partial}{\partial p_i} \Bigl( \frac{1}{\sqrt{1+ |p|^2}} \Bigr)= - \frac{p_i}{\sqrt{1+ |p|^2}^3}$.
As a result, the boundary integral reads
\begin{displaymath}
- \frac{1}{\tau^2} \int_{\partial \Omega} (u^{m+1}- u^m) (e^{m+1}_{h,u} - e^m_{h,u}) \, b^m \cdot n \, \rho^m_u \, do = 0,
\end{displaymath}
since $b^m \cdot n = - \frac{\nabla u^m \cdot n}{\sqrt{1+ | \nabla u^m |^2}^3} = 0$ on $\partial \Omega$ again by (\ref{strongbc}). Thus we obtain from \cite{DD00} (see top of
page 356) that
\begin{align}
& | A_3 | \leq \frac{1}{\tau} \int_{\Omega} d^{m+1} \cdot \nabla e^{m+1}_{h,u} \, \rho^{m+1}_u \, dx - \frac{1}{\tau} \int_{\Omega} d^m \cdot \nabla e^m_{h,u} \, \rho^m_u \, dx 
\label{err6} \\
& + \frac{6 \delta}{\tau^2} \int_{\Omega} \frac{( e^{m+1}_u - e^m_u)^2}{Q(u^m_h)} \, dx + \frac{2 \delta}{\tau} \int_{\Omega} \frac{| \nabla (e^{m+1}_{h,u} - e^m_{h,u})|^2}{Q(\hat u^m_h)} \, dx
+ \frac{c}{\delta} h^4 | \log h |^4 + \frac{c}{\delta} \bigl( E^m + E^{m+1} \bigr) \nonumber
\end{align}
with $d^m$ as in (\ref{defdm}) (see top of page 355). Next, (\ref{rmest}) implies that
\begin{align}
 | A_4| &\leq \frac{1}{\tau} \Vert R^m \Vert \Vert e^{m+1}_{h,u} - e^m_{h,u} \Vert \leq \frac{c}{\tau} \bigl( \tau + h^2 \bigr) \bigl( \Vert e^{m+1}_u - e^m_u \Vert+ 
\Vert \rho^{m+1}_u - \rho^m_u \Vert \bigr) \nonumber \\
& \leq \frac{\delta}{\tau^2} \int_{\Omega} \frac{( e^{m+1}_u - e^m_u)^2}{Q(u^m_h)} \, dx + \frac{c}{\delta} \bigl( \tau^2 + h^4 | \log h |^4 \bigr), \label{err7}
\end{align}
since $\displaystyle \Vert \rho^{m+1}_u - \rho^m_u \Vert \leq c \tau \sup_{t_m \leq t \leq t_{m+1}} \Vert u_t(\cdot,t) - \hat u_{h,t}(\cdot,t) \Vert \leq c \tau h^2 | \log h |^2$
by (\ref{ms3}). Recalling (\ref{whbound}) and the assumption that $f \in C^{0,1}_{\mbox{\footnotesize{loc}}}(\mathbb R)$ we obtain in a similar way
\begin{align}
| A_5 | \leq \frac{c}{\tau} \int_{\Omega} | e^m_w | \, | e^{m+1}_{h,u} - e^m_{h,u} | \, dx \leq \frac{\delta}{\tau^2} \int_{\Omega} \frac{( e^{m+1}_u - e^m_u)^2}{Q(u^m_h)} \, dx
+ \frac{c}{\delta} \bigl( \Vert e^m_w \Vert^2 + h^4 | \log h|^4 \bigr). \label{err8}
\end{align}
If we insert (\ref{err3})--(\ref{err8}) into (\ref{err2}) we obtain after multiplying by $\tau$ and choosing $\delta>0$ sufficiently small
\begin{align*}
& \frac{1}{2 \tau} \int_{\Omega} \frac{( e^{m+1}_u - e^m_u)^2}{Q(u^m_h)} \, dx + \frac{1}{8} \int_{\Omega} \frac{| \nabla (e^{m+1}_{h,u} - e^m_{h,u})|^2}{Q(\hat u^m_h)} \, dx 
 + (\frac{1}{2} - c\tau )E^{m+1} - \int_{\Omega} d^{m+1} \cdot \nabla e^{m+1}_{h,u} \, \rho^{m+1}_u \, dx \\
& \leq (\frac{1}{2} + c \tau) E^m - \int_{\Omega} d^m \cdot \nabla e^m_{h,u} \, \rho^m_u \, dx + c \tau (\tau^2 + h^4 | \log h |^4) + c \tau \Vert e^m_w \Vert^2.
\end{align*}
Recalling the definition of $F^m$ (\ref{defF}), and noting (\ref{embound1}) and (\ref{ms1}), we deduce that
\begin{align}
& \frac{1}{2 \tau} \int_{\Omega} \frac{( e^{m+1}_u - e^m_u)^2}{Q(u^m_h)} \, dx + \frac{1}{8} \int_{\Omega} \frac{| \nabla (e^{m+1}_{h,u} - e^m_{h,u})|^2}{Q(\hat u^m_h)} \, dx 
 + (1 - c\tau )F^{m+1} \nonumber \\
& \leq (1 + c \tau) F^m + c \tau h^2 \bigl( \Vert \nabla e^{m+1}_{h,u} \Vert + \Vert \nabla e^m_{h,u} \Vert \bigr) 
+ c \tau (\tau^2 + h^4 | \log h |^4) + c \tau \Vert e^m_w \Vert^2. \label{err9}
\end{align}
The second term on the right hand side of (\ref{err9}) is estimated by
\begin{align*}
& \tau h^2  \bigl( \Vert \nabla e^{m+1}_{h,u} \Vert + \Vert \nabla e^m_{h,u} \Vert \bigr) \leq \tau h^2 \bigl( \Vert \nabla (e^{m+1}_{h,u} -e^m_{h,u})  \Vert
+ 2 \Vert \nabla e^m_{h,u} \Vert \bigr) \\
& \leq  \frac{1}{16} \int_{\Omega} \frac{| \nabla (e^{m+1}_{h,u} - e^m_{h,u})|^2}{Q(\hat u^m_h)} \, dx + c \tau \Vert \nabla e^m_{h,u} \Vert^2 + c \tau h^4 \\
& \leq  \frac{1}{16} \int_{\Omega} \frac{| \nabla (e^{m+1}_{h,u} - e^m_{h,u})|^2}{Q(\hat u^m_h)} \, dx + c \tau F^m + c \tau h^4,
\end{align*}
where we have used (\ref{embound2}) in the last step. 
Inserting this estimate into (\ref{err9}) we infer that
\begin{eqnarray} 
\lefteqn{
\frac{1}{2 \tau} \int_{\Omega} \frac{( e^{m+1}_u - e^m_u)^2}{Q(u^m_h)} \, dx + F^{m+1} +\frac{1}{16} \int_{\Omega} \frac{| \nabla (e^{m+1}_{h,u} - e^m_{h,u})|^2}{Q(\hat u^m_h)} \, dx }
\nonumber \\
& \leq &  (1+ c \tau )F^m + c \tau (\tau^2 + h^4 | \log h |^4) + c \tau 
\int_{\Omega} (e^m_{h,w})^2 Q(u^m_h) \, dx. \label{fmp1}
\end{eqnarray}
We deduce from (\ref{fmp1}) and the induction hypothesis (\ref{induction}) 
together with (\ref{qbound}) 
\begin{eqnarray}
 \frac{1}{2c} \frac{1}{\tau} \Vert e^{m+1}_u - e^m_u \Vert^2+ F^{m+1}  
& \leq  & \bigl( \tau^2 + h^4 | \log h |^4 \bigr) e^{\mu t_m} \Bigl(  (1 +   c \tau \bigl( 1 +   \frac{1}{\beta^2 } \bigr) \Bigr)  \nonumber \\
& \leq &  \bigl( \tau^2 + h^4 | \log h |^4 \bigr)  e^{\mu t_{m+1}} \leq h^2 | \log h |^{-1}, \label{fmp2}
\end{eqnarray}
provided that 
\begin{equation} \label{cond2}
\mu \geq c \bigl( 1 + \frac{1}{\beta^2} \bigr).
\end{equation}
Note that for the last inequality in (\ref{fmp2}) we have used again (\ref{defh0}), (\ref{cond1}) and the fact that $\tau \leq \delta_0 h | \log h |^{-\frac{1}{2}}$. In particular,  
we can repeat the arguments leading to (\ref{qbound}) and (\ref{embound2}) and obtain 
\begin{equation} \label{qbound1}
\sup_{\bar \Omega} Q(u^{m+1}_h) \leq c \quad \mbox{ and } \quad  \Vert \nabla e^{m+1}_{h,u} \Vert^2 \leq c h^2 | \log h |^{-1}.
\end{equation}

\subsection{The surface PDE}
\noindent
As already mentioned in the Introduction the error analysis of the surface equation is laborious. Much of this work is related to the handling of differences of
the form $Q(u^{m+1})-Q(u^{m+1}_h)$, which are typically split into $Q(u^{m+1})-Q(\hat u^{m+1}_h)$ and  $Q(\hat u^{m+1}_h)-Q(u^{m+1}_h)$. The second term can be 
bounded in terms of $\nabla e^{m+1}_{h,u}$, which is naturally controlled within our induction. On the other hand,  simply estimating  the first term by $\nabla \rho^{m+1}_u$ 
will  frequently lead to suboptimal error bounds, which are not sufficient to control the gradient of the discrete height function uniformly. 
Instead, we will try to exploit the structure of $Q(u)$ and frequently apply integration by parts to take advantage of the quadratic convergence of $\rho^{m+1}_u$. \\
Evaluating (\ref{Gamma2weak}) at $t=t_{m+1}$ and using the definition (\ref{wproj}) we obtain for $\eta_h \in X_{h0}$
\begin{eqnarray*}
\lefteqn{ \int_{\Omega} (w \, Q(u))_t (\cdot,t_{m+1}) \eta_h \, dx + \int_{\Omega} E(\nabla \hat u_h^{m+1}) \nabla  \hat w_h^{m+1} \cdot \nabla \eta_h \, dx } \nonumber \\
& = & - \int_{\Omega} \nabla u^{m+1} \cdot \nabla \eta_h \, V^{m+1} w^{m+1} \, dx +  \int_{\Omega} g(V^{m+1},  w^{m+1}) \, \eta_h \,Q(u^{m+1}) \, dx. 
\end{eqnarray*}
If we combine this relation with (\ref{disc2}) we deduce 
\begin{eqnarray}
\lefteqn{ \hspace{-1cm}
\int_{\Omega} e^{m+1}_{h,w}  \eta_h Q(u^{m+1}_h) \, dx - \int_{\Omega} e^m_{h,w} \eta_h Q(u^m_h) \, dx + \tau
\int_{\Omega} E(\nabla u^{m+1}_h) \nabla e^{m+1}_{h,w} \cdot \nabla \eta_h \, dx } \nonumber \\
& = & \int_{\Omega} \bigl( \hat w^{m+1}_h Q(u^{m+1}_h) - \hat w^m_h Q(u^m_h) - \tau (wQ(u))_t(\cdot,t_{m+1}) \bigr) \eta_h \, dx \nonumber \\
& & + \tau \int_{\Omega} \bigl( E(\nabla u^{m+1}_h) - E(\nabla \hat u^{m+1}_h) \bigr) \nabla \hat w^{m+1}_h \cdot \nabla \eta_h \, dx \nonumber \\
& & + \tau \int_{\Omega} \bigl( V^{m+1}_h w^{m+1}_h \nabla u^{m+1}_h - V^{m+1} w^{m+1} \nabla u^{m+1} \bigr) \cdot \nabla \eta_h \, dx \nonumber  \\
& & + \tau \int_{\Omega} \bigl( g(V^{m+1},w^{m+1}) Q(u^{m+1}) - g(V^{m+1}_h,w^m_h) Q(u^{m+1}_h) \bigr) \eta_h \, dx. \label{errw1}
\end{eqnarray}
Inserting $\eta_h=e^{m+1}_{h,w}$ we derive after some straightforward manipulations
\begin{eqnarray}
\lefteqn{ \hspace{-1.2cm}
\frac{1}{2} \int_{\Omega} (e^{m+1}_{h,w})^2 Q(u^{m+1}_h) \, dx + \tau 
\int_{\Omega} E(\nabla u^{m+1}_h) \nabla e^{m+1}_{h,w} \cdot \nabla e^{m+1}_{h,w} \, dx + \frac{1}{2} \int_{\Omega} (e^{m+1}_{h,w}- e^m_{h,w})^2 Q(u^m_h) \, dx } \nonumber \\
& = &  \frac{1}{2} \int_{\Omega} (e^m_{h,w})^2 Q(u^m_h) \, dx + \frac{1}{2} \int_{\Omega} (e^{m+1}_{h,w})^2 \bigl(  Q(u^m_h) - Q(u^{m+1}_h)  \bigr) \, dx \nonumber \\
& & + \int_{\Omega} \bigl( \hat w^{m+1}_h Q(u^{m+1}_h) - \hat w^m_h Q(u^m_h) - \tau (wQ(u))_t(\cdot,t_{m+1}) \bigr) e^{m+1}_{h,w} \, dx \nonumber \\
& & + \tau \int_{\Omega} \bigl( E(\nabla u^{m+1}_h) - E(\nabla \hat u^{m+1}_h) \bigr) \nabla \hat w^{m+1}_h \cdot \nabla e^{m+1}_{h,w} \, dx \nonumber \\
& & + \tau \int_{\Omega} \bigl( V^{m+1}_h w^{m+1}_h \nabla u^{m+1}_h - V^{m+1} w^{m+1} \nabla u^{m+1} \bigr) \cdot \nabla e^{m+1}_{h,w} \, dx \nonumber  \\
& & + \tau \int_{\Omega} \bigl( g(V^{m+1},w^{m+1}) Q(u^{m+1}) - g(V^{m+1}_h,w^m_h) Q(u^{m+1}_h) \bigr) e^{m+1}_{h,w} \, dx \nonumber \\
& =: & \frac{1}{2} \int_{\Omega} (e^m_{h,w})^2 Q(u^m_h) \, dx +  \sum_{i=1}^5 B_i. \label{errw2}
\end{eqnarray}
(i) Rearranging the estimate $\displaystyle \frac{\nabla u^m_h \cdot \nabla u^{m+1}_h+1}{Q(u^m_h) Q(u^{m+1}_h)} =\nu(u^m_h) \cdot \nu(u^{m+1}_h) \leq 1$ implies that
\begin{displaymath}
Q(u^m_h) - Q(u^{m+1}_h) \leq    \frac{\nabla u^m_h}{Q(u^m_h)} \cdot \nabla (u^m_h - u^{m+1}_h),
\end{displaymath}
so that
\begin{eqnarray*}
B_1 & \leq & \frac{1}{2} \int_{\Omega} (e^{m+1}_{h,w})^2 \frac{\nabla u^m_h}{Q(u^m_h)} \cdot \nabla (u^m_h - u^{m+1}_h) \, dx = 
\frac{1}{2} \int_{\Omega} (e^{m+1}_{h,w})^2 \frac{\nabla u^m}{Q(u^m)} \cdot \nabla (u^m_h - u^{m+1}_h) \, dx \\
& & + \frac{1}{2} \int_{\Omega} (e^{m+1}_{h,w})^2 \bigl( \frac{\nabla u^m_h}{Q(u^m_h)} - \frac{\nabla u^m}{Q(u^m)} \bigr) \cdot \nabla (u^m_h - u^{m+1}_h) \, dx =:
B_{1,1}+ B_{1,2}.
\end{eqnarray*}
Integration by parts along with an inverse estimate yields
\begin{eqnarray*}
B_{1,1} & = & \int_{\Omega} e^{m+1}_{h,w} \frac{\nabla e^{m+1}_{h,w} \cdot \nabla u^m}{Q(u^m)} (u^{m+1}_h - u^m_h) \, dx + \frac{1}{2} \int_{\Omega}
(e^{m+1}_{h,w})^2 \nabla \cdot \bigl( \frac{\nabla u^m}{Q(u^m)} \bigr) (u^{m+1}_h - u^m_h) \, dx \\
& \leq & c \int_{\Omega} \bigl( | e^{m+1}_{h,w} | \, | \nabla e^{m+1}_{h,w} | + | e^{m+1}_{h,w} |^2 \bigr) \bigl( | e^{m+1}_u - e^m_u | + | u^{m+1} - u^m | \bigr) \, dx \\
& \leq & c \Vert e^{m+1}_{h,w} \Vert_{0,\infty} \Vert e^{m+1}_{h,w} \Vert_1 \Vert e^{m+1}_u - e^m_u \Vert + c \tau \sup_{t_m \leq t \leq t_{m+1}} \Vert u_t \Vert_{0,\infty} 
\Vert e^{m+1}_{h,w} \Vert \Vert e^{m+1}_{h,w} \Vert_1 \\
& \leq & c h^{-1} \Vert e^{m+1}_{h,w} \Vert  \Vert e^{m+1}_{h,w} \Vert_1  \Vert e^{m+1}_u - e^m_u \Vert + c \tau \Vert e^{m+1}_{h,w} \Vert \Vert  e^{m+1}_{h,w} \Vert_1.
\end{eqnarray*}
Next, we deduce from (\ref{invers1}), (\ref{embound2}) and (\ref{ms2}) that
\begin{displaymath}
\Vert \nabla  e^m_u \Vert_{0,\infty} \leq \Vert \nabla e^m_{h,u} \Vert_{0,\infty} + \Vert \nabla \rho^m_u \Vert_{0,\infty} \leq c h^{-1} \Vert \nabla e^m_{h,u} \Vert + c h | \log h | \leq
c | \log h |^{-\frac{1}{2}}
\end{displaymath}
and therefore by (\ref{embed}), (\ref{invers}) and (\ref{ms3})
\begin{eqnarray*}
B_{1,2} & \leq & c \Vert \nabla e^m_u \Vert_{0,\infty} \int_{\Omega} (e^{m+1}_{h,w})^2 \bigl( | \nabla( e^{m+1}_{h,u} - e^m_{h,u}) | + | \nabla ( \hat u^{m+1}_h - \hat u^m_h) | 
\bigr)\, dx \\
& \leq & c | \log h |^{-\frac{1}{2}} \Vert e^{m+1}_{h,w} \Vert_{0,\infty} \Vert e^{m+1}_{h,w} \Vert \bigl( \Vert \nabla (e^{m+1}_{h,u} - e^m_{h,u}) \Vert + \tau \bigr) \\
& \leq & c h^{-1}  \Vert e^{m+1}_{h,w} \Vert_1 \Vert e^{m+1}_{h,w} \Vert  \Vert e^{m+1}_{h,u} - e^m_{h,u} \Vert   + 
c  \tau  \Vert e^{m+1}_{h,w} \Vert_1 \Vert e^{m+1}_{h,w} \Vert \\
& \leq & c h^{-1}  \Vert e^{m+1}_{h,w} \Vert_1 \Vert e^{m+1}_{h,w} \Vert  \Vert e^{m+1}_u - e^m_u \Vert   + 
c  \tau  \Vert e^{m+1}_{h,w} \Vert_1 \Vert e^{m+1}_{h,w} \Vert.
\end{eqnarray*}
Combining the above bounds we find that
\begin{equation} \label{b1}
B_1 \leq \epsilon \tau \Vert e^{m+1}_{h,w} \Vert_1^2 + c_\epsilon \tau \Vert e^{m+1}_{h,w} \Vert^2 + c_\epsilon h^{-2}  \Vert e^{m+1}_{h,w} \Vert^2 
\frac{1}{\tau} \Vert e^{m+1}_u - e^m_u \Vert^2.
\end{equation}
(ii) Let us write
\begin{eqnarray}
B_2 & = & \int_{\Omega} \bigl( w^{m+1} Q(u^{m+1}) - w^m Q(u^m) - \tau (wQ(u))_t(\cdot,t_{m+1}) \bigr) e^{m+1}_{h,w} \, dx \nonumber \\
& & - \int_{\Omega} ( \rho^{m+1}_w - \rho^m_w) Q(u^{m+1}_h) e^{m+1}_{h,w} \, dx  - \int_{\Omega} \rho^m_w \bigl( Q(u^{m+1}_h) - Q(u^m_h) \bigr) e^{m+1}_{h,w} \, dx \nonumber  \\
& & + \int_{\Omega} ( w^{m+1}- w^m) \bigl( Q(u^{m+1}_h) - Q(u^{m+1}) \bigr) e^{m+1}_{h,w} \, dx \nonumber  \\
& & + \int_{\Omega} w^m \Bigl( \bigl( Q( \hat u^{m+1}_h) - Q( u^{m+1}) \bigr) - \bigl( Q( \hat u^m_h) - Q(u^m) \bigr) \Bigr)  e^{m+1}_{h,w} \, dx \nonumber  \\
& & + \int_{\Omega} w^m \Bigl( \bigl( Q(u^{m+1}_h) - Q(\hat u^{m+1}_h) \bigr) - \bigl( Q(u^m_h) - Q(\hat u^m_h) \bigr) \Bigr) e^{m+1}_{h,w} \, dx \nonumber  \\
& =: & \sum_{j=1}^6 B_{2,j}. \label{b2a}
\end{eqnarray}
Recalling (\ref{regu2}), (\ref{regw}), (\ref{qbound1}) and (\ref{w4}) we have
\begin{displaymath}
| B_{2,1} | + | B_{2,2} | \leq c \bigl( \tau^2 + \tau  \sup_{t_m \leq t \leq t_{m+1}} \Vert \rho_{w,t} \Vert \bigr) \Vert e^{m+1}_{h,w} \Vert \leq
c \tau \bigl( \tau + h^2 | \log h |^2 \bigr) \Vert e^{m+1}_{h,w} \Vert.
\end{displaymath}
Next, since $| Q(u^{m+1}_h) - Q(u^m_h) | \leq | \nabla (u^{m+1}_h - u^m_h) |$ we obtain with the help of (\ref{w2}), (\ref{embed}), (\ref{ms1}), (\ref{invers}) and (\ref{ms3}) 
\begin{eqnarray*}
\lefteqn{ | B_{2,3} |  \leq  \Vert \rho^m_w \Vert \Vert \nabla( u^{m+1}_h - u^m_h) \Vert \Vert e^{m+1}_{h,w} \Vert_{0,\infty} } \nonumber \\
& \leq & c h^2 | \log h | \bigl( \Vert \nabla (e^{m+1}_{h,u} - e^m_{h,u} ) \Vert + \Vert \nabla (\hat u_h^{m+1}- \hat u_h^m) \Vert \bigr)
| \log h |^{\frac{1}{2}} \Vert e^{m+1}_{h,w} \Vert_1 \nonumber \\
& \leq & c h | \log h |^{\frac{3}{2}} \Vert e^{m+1}_{h,u} - e^m_{h,u} \Vert \Vert e^{m+1}_{h,w} \Vert_1 + c \tau h^2 | \log h |^{\frac{3}{2}} \Vert 
e^{m+1}_{h,w} \Vert_1 \nonumber \\ 
& \leq & c \Vert e^{m+1}_{h,w} \Vert_1 \bigl(  \Vert e^{m+1}_u - e^m_u \Vert + \tau  h^2 | \log h |^{\frac{3}{2}} \bigr). 
\end{eqnarray*}
Applying Lemma \ref{Fest} to $f= (w^{m+1}-w^m) e^{m+1}_{h,w}$ yields
\begin{eqnarray*}
B_{2,4} & = & \int_{\Omega} (w^{m+1}-w^m) \bigl( Q(u^{m+1}_h) - Q(\hat u^{m+1}_h) \bigr) e^{m+1}_{h,w} \, dx \nonumber  \\
& & + \int_{\Omega} (w^{m+1}-w^m) \bigl( Q(\hat u^{m+1}_h) - Q(u^{m+1}) 
\bigr) e^{m+1}_{h,w} \, dx \nonumber \\
& \leq & c \Vert w^{m+1} - w^m \Vert_{0,\infty}  \Vert \nabla e^{m+1}_{h,u} \Vert \Vert e^{m+1}_{h,w} \Vert + c h^2 | \log h | \Vert (w^{m+1} -w^m) e^{m+1}_{h,w} \Vert_{1,1} \nonumber  \\
& \leq & c \tau \Vert e^{m+1}_{h,w} \Vert_1 \bigl( \Vert \nabla e^{m+1}_{h,u} \Vert + h^2 | \log h | \bigr). 
\end{eqnarray*}
Since $\displaystyle ( Q(\hat u_h) - Q(u)) \bigr)_t= \frac{\nabla \hat u_{h,t} \cdot \nabla \hat u_h}{Q(\hat u_h)} -
\frac{\nabla  u_t \cdot \nabla  u}{Q(u)}$ we obtain
\begin{eqnarray}
\lefteqn{
B_{2,5}  =   \int_{t_m}^{t_{m+1}} \int_{\Omega} w^m \bigl( \frac{\nabla \hat u_{h,t} \cdot \nabla \hat u_h}{Q(\hat u_h)} -
\frac{\nabla  u_t \cdot \nabla  u}{Q(u)} \bigr) e^{m+1}_{h,w} \, dx \, dt } \label{b25} \\
& = &  \int_{t_m}^{t_{m+1}} \int_{\Omega} w^m  \nabla u_t \cdot \bigl(  \frac{\nabla \hat u_h}{Q(\hat u_h)} -
\frac{\nabla  u}{Q(u)} \bigr)  e^{m+1}_{h,w} \, dx \, dt \nonumber  \\
& &  + \int_{t_m}^{t_{m+1}} \int_{\Omega} w^m \frac{\nabla u}{Q(u)} \cdot \nabla (\hat u_{h,t} - u_t) e^{m+1}_{h,w} \, dx \, dt \nonumber  \\
& & +  \int_{t_m}^{t_{m+1}} \int_{\Omega} w^m   \nabla ( \hat u_{h,t} - u_t) \cdot \bigl(  \frac{\nabla \hat u_h}{Q(\hat u_h)} -
\frac{\nabla  u}{Q(u)} \bigr) e^{m+1}_{h,w} \, dx \, dt  =:  I+II +III. \nonumber
\end{eqnarray}
Another  application of Lemma \ref{Fest} yields
\begin{displaymath}
I  \leq  c h^2  | \log h | \int_{t_m}^{t_{m+1}}   \Vert w^m e^{m+1}_{h,w} \,  \nabla u_t \Vert_{1,1} dt 
  \leq  c \tau  h^2 | \log h | \Vert e^{m+1}_{h,w} \Vert_1.
 \end{displaymath}
 After integration by parts we obtain
\begin{eqnarray*}
II & = & \int_{t_m}^{t_{m+1}} \int_{\Omega} \nabla \cdot \bigl( w^m \frac{\nabla u}{Q(u)} \bigr) \cdot \rho_{u,t} \, e^{m+1}_{h,w} \, dx \, dt +
\int_{t_m}^{t_{m+1}} \int_{\Omega} w^m \rho_{u,t}  \frac{\nabla u}{Q(u)} \cdot \nabla e^{m+1}_{h,w} \, dx \, dt \\
& \leq & c \tau  \sup_{t_m \leq t \leq t_{m+1}} \Vert \rho_{u,t} \Vert \, \Vert e^{m+1}_{h,w}  \Vert_1 \leq c \tau  h^2 | \log h|^2 \Vert e^{m+1}_{h,w} \Vert_1
\end{eqnarray*}
by (\ref{ms3}). Next, (\ref{ms2}) and (\ref{ms4}) imply
\begin{displaymath}
III \leq c \int_{t_m}^{t_{m+1}} \Vert \nabla \rho_{u,t} \Vert \,  \Vert \nabla \rho_u \Vert_{0,\infty}  \Vert e^{m+1}_{h,w} \Vert dt  \leq c \tau  h^2 | \log h | \Vert
 e^{m+1}_{h,w} \Vert.
\end{displaymath}
If we insert the above estimates into (\ref{b25}) we obtain 
\begin{displaymath}
B_{2,5} \leq c \tau  h^2 | \log h |^2 \Vert e^{m+1}_{h,w} \Vert_1.
\end{displaymath}
In order to treat $B_{2,6}$ we write with the help of (\ref{q2})
\begin{eqnarray}
\lefteqn{
\bigl( Q(u^{m+1}_h) - Q(\hat u^{m+1}_h) \bigr) - \bigl( Q(u^m_h) - Q(\hat u^m_h) \bigr) = - \frac{\nabla \hat u^{m+1}_h}{Q(\hat u^{m+1}_h)} \cdot \nabla e^{m+1}_{h,u}
+ \frac{\nabla \hat u^m_h}{Q(\hat u^m_h)} \cdot \nabla e^m_{h,u} } \nonumber \\
&  &   +
\frac{| \nabla e^{m+1}_{h,u} |^2}{2 Q(\hat u^{m+1}_h)} - \frac{| \nabla e^m_{h,u} |^2}{2 Q(\hat u^m_h)} - \frac{\bigl( Q(u^{m+1}_h) - Q(\hat u^{m+1}_h) \bigr)^2}{2 Q(\hat u^{m+1}_h)}
+ \frac{\bigl( Q(u^m_h) - Q(\hat u^m_h) \bigr)^2}{2 Q(\hat u^m_h)} \nonumber \\
& = & - \frac{\nabla \hat u^{m+1}_h}{Q(\nabla \hat u^{m+1}_h)} \cdot \nabla (e^{m+1}_{h,u} - e^m_{h,u} \bigr) - \bigl( \frac{\nabla \hat u^{m+1}_h}{Q(\hat u^{m+1}_h)} -
\frac{\nabla \hat u^m_h}{Q(\hat u^m_h)} \bigr) \cdot \nabla e^m_{h,u}  \label{Qdif} \\
& & + \frac{1}{2} \Bigl( \frac{1}{Q(\hat u^{m+1}_h)} - \frac{1}{Q(\hat u^m_h)} \Bigr) \Bigl( | \nabla e^m_{h,u} |^2 - \bigl( Q(u^m_h) - Q(\hat u^m_h) \bigr)^2 \Bigr) \nonumber \\
& & + \frac{ \nabla (e^{m+1}_{h,u} - e^m_{h,u}) \cdot \nabla (e^{m+1}_{h,u} + \nabla e^m_{h,u})}{2 Q(\hat u^{m+1}_h)} - 
\delta_h \{ \bigl( Q(u^{m+1}_h) - Q(\hat u^{m+1}_h) \bigr) - \bigl( Q(u^m_h) - Q(\hat u^m_h) \bigr) \}, \nonumber
\end{eqnarray}
where
\begin{displaymath}
\delta_h = \frac{\bigl( Q(u^{m+1}_h)-Q(\hat u^{m+1}_h) \bigr) + \bigl( Q(u^m_h) - Q(\hat u^m_h) \bigr)}{2 Q(\hat u^{m+1}_h)}.
\end{displaymath}
We remark that (\ref{invers1}), (\ref{embound2}) and (\ref{qbound1}) imply that
\begin{equation} \label{qhest}
| \delta_h | \leq \frac{1}{2} \bigl( | \nabla e^{m+1}_{h,u}| + | \nabla e^m_{h,u} | \bigr) \leq c h^{-1} \bigl( \Vert \nabla e^{m+1}_{h,u} \Vert + \Vert \nabla e^m_{h,u} \Vert 
\bigr) \leq c | \log h |^{-\frac{1}{2}} \leq \frac{1}{2},
\end{equation}
provided that $0 < h \leq h_3$ and $h_3 \leq h_2$  is small enough. Thus, if we move the last term on the right hand side of (\ref{Qdif}) to the left hand side
and divide by $1+\delta_h \geq \frac{1}{2}$ we end up with
\begin{eqnarray*}
\lefteqn{ \hspace{-1cm}
\bigl( Q(u^{m+1}_h) - Q(\hat u^{m+1}_h) \bigr) - \bigl( Q(u^m_h) - Q(\hat u^m_h) \bigr) =
- \frac{\nabla \hat u^{m+1}_h}{Q(\nabla \hat u^{m+1}_h)} \cdot \nabla (e^{m+1}_{h,u} - e^m_{h,u} \bigr) } \\
 & & 
 + \frac{1}{1+ \delta_h} \nabla (e^{m+1}_{h,u} - e^m_{h,u}) \cdot \Bigl( \delta_h \,  \frac{\nabla \hat u^{m+1}_h}{Q(\nabla \hat u^{m+1}_h)} +
 \frac{\nabla (e^{m+1}_{h,u} + \nabla e^m_{h,u})}{2 Q(\hat u^{m+1}_h)} \Bigr) \\
 & & + \frac{1}{1+ \delta_h} \frac{Q(\hat u^m_h) - Q(\hat u^{m+1}_h)}{2 Q(\hat u^{m+1}_h) Q(\hat u^m_h)}  \Bigl( | \nabla e^m_{h,u} |^2 - \bigl( Q(u^m_h) - Q(\hat u^m_h) \bigr)^2 \Bigr) \\
 & & - \frac{1}{1+ \delta_h}  \bigl( \frac{\nabla \hat u^{m+1}_h}{Q(\hat u^{m+1}_h)} -
\frac{\nabla \hat u^m_h}{Q(\hat u^m_h)} \bigr) \cdot \nabla e^m_{h,u}  =:  S_1 + S_2 + S_3+S_4.
\end{eqnarray*}
If we insert this expression into $B_{2,6}$ we obtain
\begin{displaymath}
B_{2,6} = \sum_{i=1}^4 \int_{\Omega} w^m S_i e^{m+1}_{h,w} \, dx.
\end{displaymath}
To begin, integration by parts together with (\ref{strongbc}) yields
\begin{eqnarray*}
\int_{\Omega} w^m S_1 e^{m+1}_{h,w} \, dx  &= & \int_{\Omega} w^m e^{m+1}_{h,w} \left[  \bigl( \frac{\nabla u^{m+1}}{Q(u^{m+1})} 
- \frac{\nabla \hat u^{m+1}_h}{Q(\hat u^{m+1}_h)} \bigr) 
 - \frac{\nabla u^{m+1}}{Q(u^{m+1})} \right] \cdot \nabla (e^{m+1}_{h,u} - e^m_{h,u}) \, dx  \\
 & = &  \int_{\Omega} w^m e^{m+1}_{h,w}  \bigl( \frac{\nabla u^{m+1}}{Q(u^{m+1})} - \frac{\nabla \hat u^{m+1}_h}{Q(\hat u^{m+1}_h)} \bigr) 
  \cdot \nabla (e^{m+1}_{h,u} - e^m_{h,u}) \, dx \\
  & & + \int_{\Omega}  \nabla \cdot \bigl( \frac{ w^m e^{m+1}_{h,w} \, \nabla u^{m+1}}{Q(u^{m+1})} \bigr) (e^{m+1}_{h,u} - e^m_{h,u}) \, dx.
\end{eqnarray*}
Using Lemma \ref{Fest1} and (\ref{ms3}) for the first term we obtain
\begin{eqnarray*}
\lefteqn{ \hspace{-2.5cm}
\int_{\Omega} w^m S_1 e^{m+1}_{h,w} \, dx \leq c h | \log h | \Vert e^{m+1}_{h,u} - e^m_{h,u} \Vert \Bigl( \sum_{T \in \mathcal T_h} \Vert w^m e^{m+1}_{h,w} \Vert_{H^2(T)}^2 \Bigr)^{\frac{1}{2}} + c \Vert e^{m+1}_{h,u} - e^m_{h,u} \Vert  \Vert e^{m+1}_{h,w} \Vert_1 } \\
& \leq &  c h | \log h | \Vert  e^{m+1}_{h,w}  \Vert_1 \Vert  e^{m+1}_{h,u} - e^m_{h,u} \Vert \leq c \Vert e^{m+1}_{h,w} \Vert_1 \Vert e^{m+1}_{h,u} - e^m_{h,u} \Vert.
\end{eqnarray*}
Since $1+\delta_h \geq \frac{1}{2}$ and $| \delta_h | \leq \frac{1}{2} \bigl( | \nabla e^{m+1}_{h,u} | + | \nabla e^m_{h,u} | \bigr)$ we derive with the help of (\ref{embound2}), (\ref{qbound1}), (\ref{invers}) and (\ref{embed})
\begin{eqnarray*}
\lefteqn{
\int_{\Omega} w^m S_2 e^{m+1}_{h,w} \, dx \leq c \Vert \nabla (e^{m+1}_{h,u} - e^m_{h,u}) \Vert \bigl( \Vert \nabla e^{m+1}_{h,u} \Vert + \Vert \nabla e^m_{h,u} \Vert \bigr)
\Vert e^{m+1}_{h,w} \Vert_{0,\infty} } \\
& \leq & c h | \log h |^{-\frac{1}{2}} h^{-1} \Vert e^{m+1}_{h,u} - e^m_{h,u} \Vert \, | \log h |^{\frac{1}{2}} \Vert e^{m+1}_{h,w} \Vert_1 \leq
c \Vert e^{m+1}_{h,u} - e^m_{h,u}  \Vert \Vert e^{m+1}_{h,w} \Vert_1.
\end{eqnarray*}
Finally, we deduce with the help of (\ref{qhest}), (\ref{embed}) and (\ref{embound2})
\begin{eqnarray*}
\lefteqn{ \hspace{-1cm}
\int_{\Omega} w^m (S_3+ S_4) e^{m+1}_{h,w} \, dx \leq  c \int_{\Omega} | \nabla (\hat u^{m+1}_h - \hat u^m_h)| \bigl( | \nabla e^m_{h,u} |^2 + | \nabla e^m_{h,u} | \bigr) 
| e^{m+1}_{h,w} | \, dx }  \\
& \leq & c \Vert \nabla ( \hat u^{m+1}_h - \hat u^m_h) \Vert_{0,\infty} \Vert \nabla e^m_{h,u} \Vert \bigl( \Vert \nabla e^m_{h,u} \Vert \Vert e^{m+1}_{h,w} \Vert_{0,\infty} 
+ \Vert e^{m+1}_{h,w} \Vert  \bigr) \\
& \leq &    c \tau \Vert \nabla e^m_{h,u} \Vert \bigl(  h \Vert e^{m+1}_{h,w} \Vert_1 + \Vert e^{m+1}_{h,w} \Vert \bigr)  \leq   c \tau \Vert \nabla e^m_{h,u} \Vert \Vert e^{m+1}_{h,w} \Vert_1.
\end{eqnarray*}
Collecting the above estimates and recalling (\ref{ms3}) we obtain
\begin{displaymath}
B_{2,6}  \leq  c \Vert e^{m+1}_{h,w}  \Vert_1 \bigl( \Vert e^{m+1}_u - e^m_u \Vert + \tau h^2 | \log h |^2 + \tau \Vert \nabla e^m_{h,u} \Vert \bigr).
\end{displaymath}
If we insert the bounds for $B_{2,j}, j=1,\ldots,6$ into (\ref{b2a}) we obtain
\begin{eqnarray}
B_2 &\leq &  c \Vert e^{m+1}_{h,w} \Vert_1 \bigl( \tau^2 + \tau h^2 | \log h |^2 + \Vert e^{m+1}_u - e^m_u \Vert + \tau \Vert \nabla e^{m+1}_{h,u} \Vert +
\tau \Vert \nabla e^m_{h,u} \Vert \bigr) \label{b2}  \\
& \leq & \epsilon \tau \Vert e^{m+1}_{h,w} \Vert_1^2 + c_\epsilon \frac{1}{\tau} \Vert e^{m+1}_u - e^m_u \Vert^2
+ c_\epsilon \tau \bigl( \tau^2 + h^4 | \log h |^4  + \Vert \nabla e^{m+1}_{h,u} \Vert^2 + \Vert \nabla e^m_{h,u} \Vert^2 \bigr). \nonumber
\end{eqnarray}
(iii)  Recalling (\ref{defE}) it is not difficult to verify  that  $|E(p) - E(q) | \leq c | p-q |$ and hence
\begin{equation} \label{b3}
B_3 \leq C \tau \Vert \nabla e^{m+1}_{h,u} \Vert \Vert \nabla e^{m+1}_{h,w} \Vert \leq \epsilon \tau \Vert \nabla e^{m+1}_{h,w} \Vert^2 + c_\epsilon \tau
\Vert \nabla e^{m+1}_{h,u} \Vert^2.
\end{equation}
(iv) In view of the definition of $V^{m+1}$ and $V_h^{m+1}$ we have 
\begin{eqnarray*}
\lefteqn{ B_4  =  \int_{\Omega} \Bigl( (u^{m+1}_h - u^m_h) w^m_h \frac{\nabla u^{m+1}_h}{Q(u^{m+1}_h)} - \tau u^{m+1}_t w^{m+1} \frac{ \nabla u^{m+1}}{Q(u^{m+1})} \Bigr)
\cdot \nabla e^{m+1}_{h,w} \, dx } \\
& = & - \int_{\Omega} (e^{m+1}_u - e^m_u) w^m_h \frac{\nabla u^{m+1}_h}{Q(u^{m+1}_h)} \cdot \nabla e^{m+1}_{h,w} \, dx - \int_{\Omega} (u^{m+1}-u^m)
e^m_{w} \frac{\nabla u^{m+1}_h}{Q(u^{m+1}_h)} \cdot \nabla e^{m+1}_{h,w} \, dx \\
& & + \int_{\Omega} (u^{m+1}-u^m) w^m \bigl( \frac{\nabla u^{m+1}_h}{Q(u^{m+1}_h)} - \frac{\nabla \hat u^{m+1}_h}{Q(\hat u^{m+1}_h)} \bigr) \cdot \nabla
e^{m+1}_{h,w} \, dx \\
& & + \int_{\Omega} (u^{m+1} - u^m) w^m \bigl( \frac{\nabla \hat u^{m+1}_h}{Q(\hat u^{m+1}_h)} - \frac{\nabla u^{m+1}}{Q(u^{m+1})} \bigr)
\cdot \nabla e^{m+1}_{h,w} \, dx \\
& & + \int_{\Omega} \bigl( (u^{m+1}- u^m - \tau u^{m+1}_t ) w^m + \tau u^{m+1}_t (w^m-w^{m+1}) \bigr) \frac{\nabla u^{m+1}}{Q(u^{m+1})} \cdot \nabla e^{m+1}_{h,w} \, dx =: \sum_{i=1}^5 B_{4,i}.
\end{eqnarray*}
It follows from (\ref{whbound}) and (\ref{rmest}) that
\begin{displaymath}
B_{4,1} + B_{4,2} + B_{4,3} + B_{4,5} 
 \leq  c \Vert \nabla e^{m+1}_{h,w} \Vert \bigl( \Vert e^{m+1}_u - e^m_u \Vert + \tau \Vert e^m_w \Vert + \tau \Vert \nabla e^{m+1}_{h,u} \Vert + \tau^2 \bigr),
\end{displaymath}
while Lemma \ref{Fest1} implies
\begin{displaymath}
B_{4,4} \leq c h^2 | \log h | \Vert e^{m+1}_{h,w} \Vert_1 \Vert (u^{m+1}-u^m) w^m \Vert_2 \leq c \tau h^2 | \log h | \Vert e^{m+1}_{h,w} \Vert_1.
\end{displaymath}
In conclusion
\begin{equation} \label{b4}
B_4 \leq \epsilon \tau \Vert e^{m+1}_{h,w} \Vert_1^2 + c_\epsilon \frac{1}{\tau} \Vert e^{m+1}_u - e^m_u \Vert^2 + c_\epsilon \tau \bigl( \Vert e^m_{h,w} \Vert^2
+ \Vert \nabla e^{m+1}_{h,u} \Vert^2 + \tau^2 + h^4 | \log h |^2 \bigr).
\end{equation}
(v) Finally, in order to treat $B_5$  we recall (\ref{formg}) and note that $\alpha(\lambda r) = \lambda \alpha(r), r \in \mathbb R, \lambda>0$. As a
consequence, 
\begin{eqnarray*}
\lefteqn{
g(V^{m+1},w^{m+1}) Q(u^{m+1}) - g(V^{m+1}_h,w^m_h) Q(u^{m+1}_h) } \\
& = & \alpha( u^{m+1}_t) \bigl( \beta(w^{m+1}) - \beta(w^m_h) \bigr) + \Bigl( \alpha(u^{m+1}_t) - \alpha \bigl( \frac{u^{m+1}_h -u^m_h}{\tau} \bigr) \Bigr) \beta(w^m_h) \\
& & + \tilde \beta (w^{m+1}) \bigl( Q(u^{m+1}) - Q(u^{m+1}_h) \bigr) + \bigl( \tilde \beta(w^{m+1}) - \tilde \beta( w^m_h) \bigr) Q(u^{m+1}_h)
= \sum_{i=1}^4 S_i,
\end{eqnarray*}
so that 
\begin{displaymath}
B_5 = \sum_{i=1}^4 \tau \int_{\Omega} S_i \,  e^{m+1}_{h,w} \, dx =: \sum_{i=1}^4 B_{5,i}.
\end{displaymath}
Since $\beta, \tilde \beta \in C^{0,1}_{\mbox{\footnotesize{loc}}}(\mathbb R)$ we obtain from (\ref{whbound}), (\ref{qbound1}) and (\ref{w2}) 
\begin{eqnarray*}
| B_{5,1} | + | B_{5,4} |  & \leq &  c \tau \Vert w^{m+1}- w^m_h \Vert \Vert e^{m+1}_{h,w} \Vert \leq c \tau \bigl( \Vert w^{m+1} - w^m \Vert + \Vert \rho^m_w \Vert + \Vert e^m_{h,w}
\Vert \bigr) \Vert e^{m+1}_{h,w} \Vert \\
& \leq &  c \tau \bigl( \tau +  h^2 | \log h | + \Vert e^m_{h,w} \Vert \bigr) \Vert e^{m+1}_{h,w} \Vert.
\end{eqnarray*}
Next, we deduce with the help of the global Lipschitz continuity of $r \mapsto \alpha(r)$ and (\ref{rmest}) that
\begin{displaymath}
| B_{5,2} |  \leq   c \tau \Vert u^{m+1}_t - \frac{u^{m+1}_h - u^m_h}{\tau} \Vert \Vert e^{m+1}_{h,w} \Vert 
 \leq  c \tau^2 \Vert e^{m+1}_{h,w} \Vert + c \Vert e^{m+1}_u - e^m_u \Vert \Vert e^{m+1}_{h,w} \Vert. 
\end{displaymath}
Applying Lemma \ref{Fest} with $f=\tilde \beta(w^{m+1}) e^{m+1}_{h,w}$ we infer that
\begin{displaymath}
| B_{5,3} | \leq c \tau h^2 | \log h | \Vert \tilde \beta (w^{m+1}) e^{m+1}_{h,w} \Vert_{1,1} \leq c \tau h^2 | \log h | \Vert e^{m+1}_{h,w} \Vert_1.
\end{displaymath}
After collecting the  above estimates we obtain
\begin{equation} \label{b5}
B_5 \leq \epsilon \tau \Vert e^{m+1}_{h,w} \Vert_1^2 + c_\epsilon \frac{1}{\tau} \Vert e^{m+1}_u - e^m_u \Vert^2 + c_\epsilon \tau \bigl( \Vert e^m_{h,w} \Vert^2
+  \tau^2 + h^4 | \log h |^2 \bigr).
\end{equation}
If we insert (\ref{b1}), (\ref{b2}), (\ref{b3}), (\ref{b4}) and (\ref{b5}) into (\ref{errw2}), use Poincar\'e's inquality and observe (\ref{ellipt}) together with (\ref{qbound1}) we derive
\begin{eqnarray}
\lefteqn{ \hspace{-0.5cm}
\frac{1}{2} \int_{\Omega} (e^{m+1}_{h,w})^2 Q(u^{m+1}_h) \, dx + \tau c_0 \Vert \nabla e^{m+1}_{h,w} \Vert^2 
+ \frac{1}{2} \Vert e^{m+1}_{h,w}- e^m_{h,w} \Vert^2 } \nonumber \\
& \leq &  \frac{1}{2} \int_{\Omega} (e^m_{h,w})^2 Q(u^m_h) \, dx + \epsilon \tau \Vert \nabla e^{m+1}_{h,w} \Vert^2 + 
c_\epsilon \tau \bigl( \tau^2 + h^4 | \log h |^4  + \Vert \nabla e^{m+1}_{h,u} \Vert^2 + \Vert \nabla e^m_{h,u} \Vert^2 \bigr) \nonumber \\
& & + c_\epsilon \tau \Vert e^m_{h,w} \Vert^2 + c_\epsilon h^{-2} \frac{1}{\tau} \Vert e^{m+1}_{h,w} \Vert^2  \Vert e^{m+1}_u - e^m_u \Vert^2  
+ c_\epsilon \frac{1}{\tau} \Vert e^{m+1}_u - e^m_u \Vert^2. \label{errw3}
\end{eqnarray}
In view of (\ref{embound}), (\ref{fmp2}) and (\ref{defh0}) we have
\begin{eqnarray*} 
h^{-2}  \Vert e^{m+1}_{h,w} \Vert^2  \frac{1}{\tau} \Vert e^{m+1}_u - e^m_u \Vert^2  & \leq & 2 h^{-2}  \bigl( \Vert e^m_{h,w} \Vert^2+
\Vert e^{m+1}_{h,w} - e^m_{h,w} \Vert^2 \bigr) \frac{1}{\tau} \Vert e^{m+1}_u - e^m_u \Vert^2  \\
& \leq & 4 \frac{| \log h |^{-1}}{\beta^2}  \frac{1}{\tau} \Vert e^{m+1}_u - e^m_u \Vert^2 + c | \log h|^{-1}  \Vert e^{m+1}_{h,w} - e^m_{h,w} \Vert^2 \\
& \leq & c  \frac{1}{\tau} \Vert e^{m+1}_u - e^m_u \Vert^2 + c | \log h|^{-1}  \Vert e^{m+1}_{h,w} - e^m_{h,w} \Vert^2.
\end{eqnarray*}
Using this bound in (\ref{errw3}) and  choosing $\epsilon$ and $h_0$ sufficiently small we obtain
\begin{eqnarray}
\lefteqn{ \hspace{-1cm}
\frac{1}{2}  \int_{\Omega} (e^{m+1}_{h,w})^2 Q(u^{m+1}_h) \, dx + \tau \frac{c_0}{2}  \Vert \nabla e^{m+1}_{h,w} \Vert^2 
\leq    \frac{1}{2}  \int_{\Omega} (e^m_{h,w})^2 Q(u^m_h) \, dx + c \tau \Vert e^m_{h,w} \Vert^2 } \label{errw4} \\ 
& &  + c \tau \bigl( \tau^2 + h^4 | \log h |^4 \bigr) + c \tau \Vert \nabla ( e^{m+1}_{h,u} - e^m_{h,u}) \Vert^2 + c \tau \Vert \nabla e^m_{h,u} \Vert^2       
+  \frac{c}{\tau} \Vert e^{m+1}_u - e^m_u \Vert^2 \nonumber  \\
& \leq  & (1+ c \tau) \frac{1}{2}  \int_{\Omega} (e^m_{h,w})^2 Q(u^m_h) \, dx + c \tau \bigl( \tau^2 +  h^4 | \log h |^4 \bigr) + c \tau F^m \nonumber \\
& & + c \tau \int_{\Omega} \frac{ | \nabla (e^{m+1}_{h,u} - e^m_{h,u}) |^2}{Q(\hat u^m_h)} \, dx       
+ c  \frac{1}{2 \tau} \int_{\Omega} \frac{(e^{m+1}_u - e^m_u)^2}{Q(u^m_h)} \, dx, \nonumber
\end{eqnarray}
where we used (\ref{embound2}) and the fact that $Q(\hat u^m_h), Q(u^m_h) \leq c$ in order to derive the last estimate.
Multiplying (\ref{errw4}) by $\beta^2$ ($0< \beta \leq 1$) and adding the result to (\ref{fmp1}) we obtain with the help of our induction hypothesis (\ref{induction})
\begin{eqnarray}
\lefteqn{ F^{m+1} + \frac{\beta^2}{2} \int_{\Omega} (e^{m+1}_{h,w})^2 Q(u^{m+1}_h) \, dx  + \tau \frac{\beta^2 c_0}{2}  \Vert \nabla e^{m+1}_{h,w} \Vert^2} \nonumber  \\
& & + (1 - c \beta^2) \frac{1}{2 \tau} \int_{\Omega} \frac{(e^{m+1}_u - e^m_u)^2}{Q(u^m_h)} \, dx 
+ ( \frac{1}{16} - c \tau) \int_{\Omega} \frac{ | \nabla (e^{m+1}_{h,u} - e^m_{h,u}) |^2}{Q(\hat u^m_h)} \, dx \nonumber  \\
& \leq & (1 + \frac{c}{\beta^2} \tau) \bigl( F^m + \frac{\beta^2}{2} \int_{\Omega} (e^m_{h,w})^2 Q(u^m_h) \, dx \bigr) + c \tau  \bigl( \tau^2+  \tau h^4 | \log h |^4 \bigr) \nonumber \\
& \leq & \Bigl( 1 + c \tau \bigl( 1 + \frac{1}{\beta^2} \bigr) \Bigr) \bigl( \tau^2 +  h^4 | \log h |^4 \bigr)  e^{\mu t_m}  \leq \bigl( \tau^2 +  h^4 | \log h |^4
\bigr) e^{\mu t_{m+1}}, \label{errw5}
\end{eqnarray}
provided that 
\begin{equation} \label{cond3}
\mu \geq c (1 + \frac{1}{\beta^2}).
\end{equation}
We are now in position to specify the choice of the constants $\beta, \mu, \delta_0$ and $h_0$. To begin, choose $0< \beta \leq 1$ such that $1 - c \beta^2 \geq \frac{1}{2}$
in the second line of (\ref{errw5}). Next choose $\mu>0$ to satisfy (\ref{cond2}) and (\ref{cond3}) and then $\delta_0>0$ to satisfy (\ref{cond1}). Finally, $h_0>0$ is
fixed by (\ref{defh0}) and additional smallness conditions on $h$ that were required in the course of the calculations.

\section{Numerical Results} \label{sec:5}
\setcounter{equation}{0}

We begin this section by investigating the experimental order of convergence (eoc) of our scheme and then we display some simulations of diffusion induced grain boundary motion. Throughout the computations in this section we choose a uniform time step $\tau = h^2$. 
\subsection{Experimental order of convergence}
 We set $\Omega:=\{x\in \mathbb{R}^2 \,|\,|x|< 1\}, \; T=0.1$ and choose $f(w)=w^2$ as well as $g(V,w)=Vw$. We consider $u,w: \bar \Omega \times [0,T] \rightarrow \mathbb R$
 given by \\[2mm]
{\bf Example 1} $u(x,t)=5\sin(t)(1-|x|^2), \; w(x,t)=e^{-t}(1+|x|^2)$; \\[1.3mm]
{\bf Example 2} $u(x,t)=5\sin(t)(1+(1-|x|^2)^2); w(x,t)=e^{-t}(1+|x|^2)$ \\[2mm]
and include additional right hand sides in order for $(u,w)$ to be solutions of the corresponding PDEs, while the boundary conditions are
$u(x,t) = 0, w(x,t)=2e^{-t}$ in Example 1 and  $\frac{\partial u}{\partial n}(x,t) =0, w(x,t)=2e^{-t}$ in Example 2. 
We commence our numerical results with Figure \ref{fig:1} in which we display the solution $w_h^m$ plotted on the surface $\Gamma_h^m=\{(x,u_h^m(x))\,|\,x\in \Omega\}$, at $t^m=0$ and $t^m=0.1$, for Example 2. 
\begin{figure}[h]
\centering
\subfigure{\includegraphics[width = 0.48\textwidth]{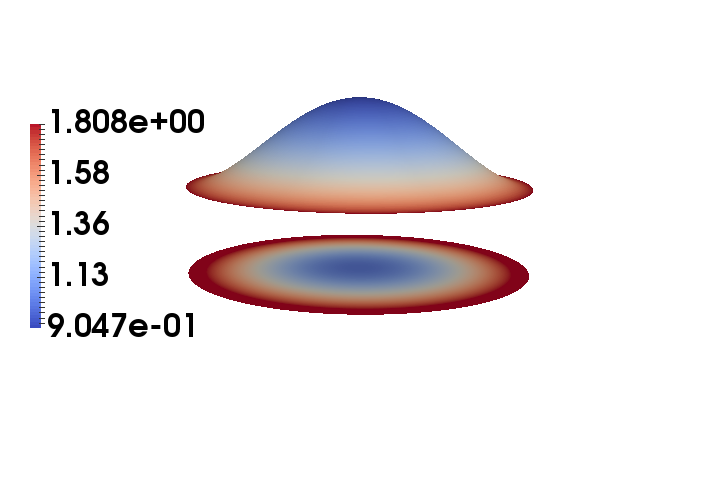}} 
\caption{{Example 2}, $w_h^m$ plotted on $\Gamma_h^m$ at $t^m=0.0$ and $t^m=0.1$.} 
\label{fig:1}
\end{figure}
\noindent
When investigating the experimental order of convergence we monitor the following errors:
$$
\mathcal{E}_1:=\max_{0 \leq m \leq M} \Vert e_w^m \Vert|^2,\;
\mathcal{E}_2:=\sum_{m=1}^M \tau \, \Vert \nabla e_w^m\Vert^2, \;
\mathcal{E}_3:=\max_{0 \leq m \leq M} \Vert e_u^m \Vert^2, \;
$$
$$
\mathcal{E}_4:=\max_{0 \leq m \leq M} \Vert \nabla e_u^m \Vert^2, \;
\mathcal{E}_5:=\sum_{m=0}^{M-1}\tau \Vert \frac{e^{m+1}_u - e^m_u}{\tau} \Vert^2.
$$
In Tables \ref{t:dir} and \ref{t:neu} we display the values of $\mathcal{E}_i$, $i=1,\ldots,5$, evaluated using a quadrature rule of degree $4$, for Example 1 and Example 2 respectively. 
For both examples we see the expected order of convergence, with eocs close to four for  $\mathcal{E}_1$, $\mathcal{E}_3$ and $\mathcal{E}_5$, and eocs close to two for $\mathcal{E}_2$ and $\mathcal{E}_3$. In particular, the  results of Example 2 confirm the bounds obtained in Theorem \ref{main}.

\begin{center}
\begin{table}[!h]
 \begin{tabular}{ |c||c|c|c|c|c|c|c|c|c|c| }
 \hline
h & $\mathcal{E}_1\times 10^2$ & $eoc_1$ & $\mathcal{E}_2\times 10^2$ & $eoc_2$ & $\mathcal{E}_4\times 10^4$ & $eoc_3$ & $\mathcal{E}_4\times 10^2$ & $eoc_4$ & $\mathcal{E}_5\times 10^4$ & $eoc_5$ \\ 
 \hline
 \hline
0.1961 & 29.33177 & - & 64.7340 & - & 81.45108 & - & 35.3598 & - & 9.88350 & -   \\ 
0.0996 & 0.06181 & 9.10 & 0.5881 & 6.94 & 0.16291 & 9.18 & 1.1420 & 5.07& 0.04050 & 8.12  \\ 
0.0538 & 0.00514 & 4.04 & 0.1006 & 2.87 & 0.00668 & 5.19 & 0.2347 & 2.57 & 0.00279 & 4.34  \\ 
0.0269 & 0.00032 & 3.99 & 0.0217 & 2.21 & 0.00037 & 4.16 & 0.0574 & 2.03 & 0.00018 & 4.00 \\ 
0.0135 & 0.00002 & 4.00 & 0.0052 & 2.06 & 0.00002 & 4.07 & 0.0142 & 2.02 & 0.00001 & 4.01 \\ \hline
\end{tabular}
\caption{{Example 1}, the errors, $\mathcal{E}_i$, together with their associated estimated order of convergence, $eoc_i$, for $i=1,2,3,4,5$.} 
\label{t:dir}
\end{table}
\end{center}

\begin{center}
\begin{table}[!h]
 \begin{tabular}{ |c||c|c|c|c|c|c|c|c|c|c| }
 \hline
h & $\mathcal{E}_1\times 10^2$ & $eoc_1$ & $\mathcal{E}_2\times 10^2$ & $eoc_2$ & $\mathcal{E}_4\times 10^4$ & $eoc_3$ & $\mathcal{E}_4\times 10^2$ & $eoc_4$ & $\mathcal{E}_5\times 10^4$ & $eoc_5$ \\ 
 \hline
 \hline
0.1961 & 38.25517 & - & 152.3575 & - & 99.40560 & - & 101.7628 & - & 38.94582 & -   \\ 
0.0996 & 0.27386 & 7.29 & 1.0056 & 7.41 & 1.72052 & 5.99 & 4.4905 & 4.61& 0.23208 & 7.56  \\ 
0.0538 & 0.02398 & 3.96 & 0.1325 & 3.29 & 0.10679 & 4.51 & 0.8869 & 2.63 & 0.01604 & 4.34  \\ 
0.0269 & 0.00153 & 3.97 & 0.0237 & 2.48 & 0.00652 & 4.03 & 0.2105 & 2.08 & 0.00099 & 4.01 \\ 
0.0135 & 0.00011 & 3.82 & 0.0054 & 2.15 & 0.00041 & 3.98 & 0.0515 & 2.03 & 0.00006 & 4.04 \\ \hline
\end{tabular}
\caption{{Example 2}, the errors, $\mathcal{E}_i$, together with their associated estimated order of convergence, $eoc_i$, for $i=1,2,3,4,5$. }
\label{t:neu}
\end{table}
\end{center}


\subsection{Non--orthogonal boundary contact}
Even though we have restricted our error analysis to the case where the evolving surface meets the boundary of the cylinder at a right angle, it is not difficult to apply our approach to the case where it meets the boundary of the cylinder at a given angle $\alpha$. In order to do so, we replace the boundary condition (\ref{bc1}) with
$$
\nu \cdot \nu_{\partial A}   = \cos(\alpha) \qquad  \mbox{on } \partial \Gamma(t), \quad t \in (0,T], \label{bca}
$$
leading to the following boundary condition for the height function $u$:
$$
\frac{ \nabla u \cdot n}{Q(u)} = -\cos(\alpha) \quad \mbox{ on } \partial \Omega \times (0,T].
$$
The weak formulation for $u$ then takes the  form
$$
\int_{\Omega} \frac{u_t \, \varphi}{Q(u)} \,  dx + \int_{\Omega} \frac{\nabla u \cdot \nabla \varphi}{Q(u)} \,  dx = \int_{\Omega} f(w) \, \varphi \, dx -  \int_{\partial\Omega} \cos(\alpha)\varphi \, dx
\qquad \forall \varphi \in H^1(\Omega),
$$
from which we derive the  corresponding finite element approximation replacing (\ref{disc1}). \\
We set $\Omega:=\{x\in \mathbb{R}^2\,|\,|x|< 1\}$, $f(w)=w$ and $g(V,w)=|V|w$ and specify the following  boundary conditions for $u$ and $w$:  
$$
\frac{\nabla u  \cdot n}{Q(u)} = -\cos(2\pi t -\pi/2)~~\mbox{and}~~w=1 \quad \mbox{ on }  \partial \Omega \times (0,T].
$$ 
As initial data we choose $u^0(x)=0$ and $w^0(x)=\frac12(1+|x|^2)$. In Figure \ref{f:angle} we display $w_h^m$ on the 
surface $\Gamma_h^m=\{(x,u_h^m(x))\,|\,x\in \Omega\}$ at $t^m=0,0.25, 0.35, 0.5,  0.65,0.75$. As $| \cos(2 \pi t - \pi/2) | =1$ for $t=0.25, 0.75$, the gradient of $u$ will blow up
on the boundary. However, for the mesh sizes we chose the discrete solution
was able to flow through these singularities without problems. 

\begin{figure}[h]
\centering
$t=0$\hspace{4cm} $t=0.25$\hspace{4cm} $t=0.35$\\
\subfigure{\includegraphics[width = 0.32\textwidth]{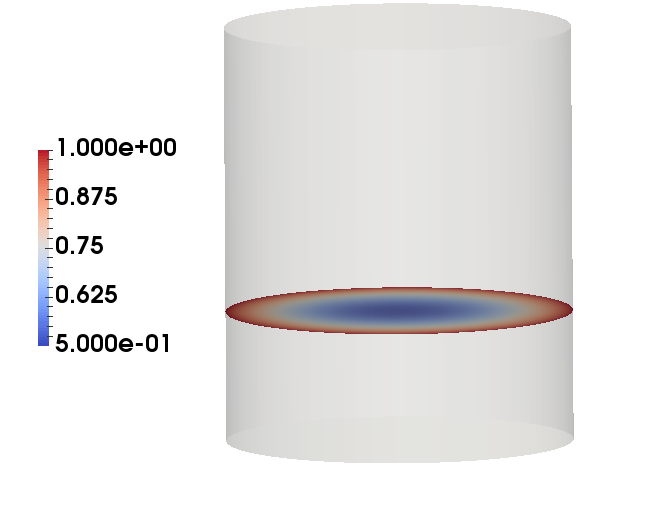}} 
\subfigure{\includegraphics[width = 0.32\textwidth]{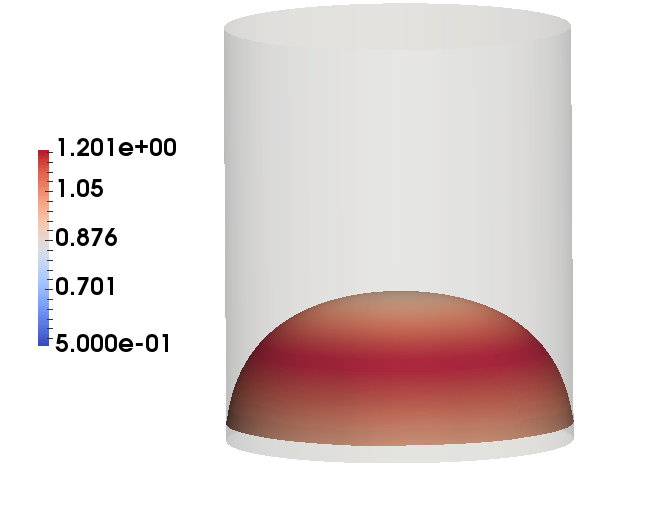}} 
\subfigure{\includegraphics[width = 0.32\textwidth]{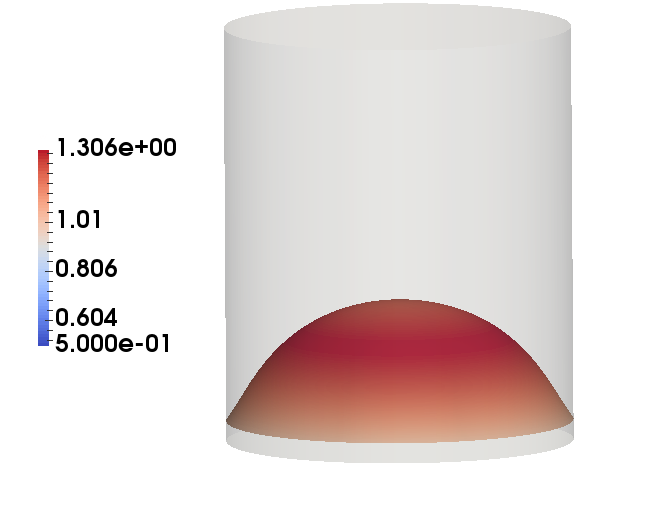}} \\[4mm]
$t=0.5$\hspace{4cm} $t=0.65$\hspace{4cm} $t=0.75$\\
\subfigure{\includegraphics[width = 0.32\textwidth]{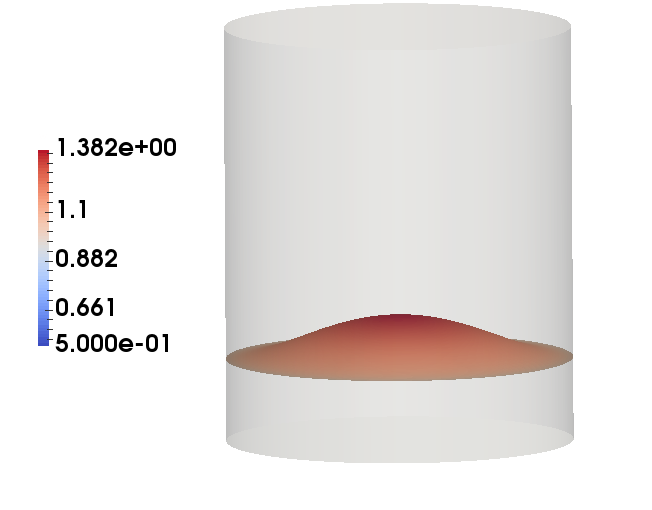}} 
\subfigure{\includegraphics[width = 0.32\textwidth]{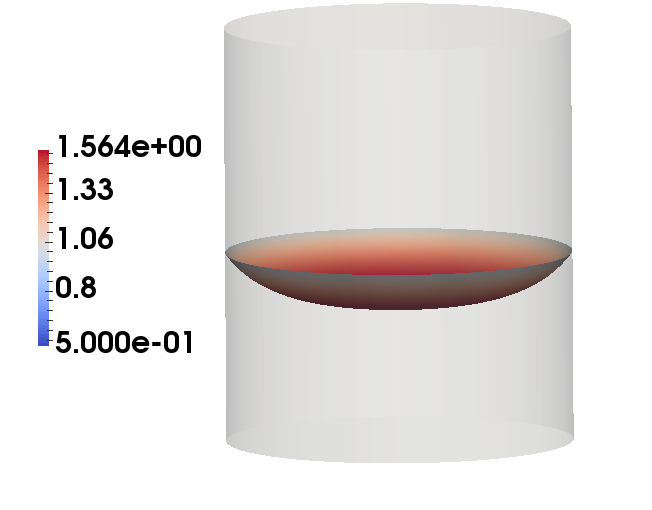}} 
\subfigure{\includegraphics[width = 0.32\textwidth]{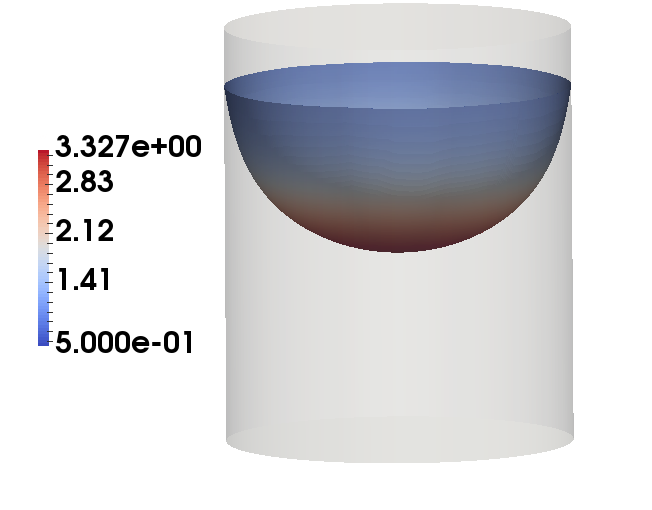}} \\
\caption{$w_h^m$ plotted on the surface $\Gamma_h^m$ at $t^m=0,0.25, 0.35, 0.5,  0.65,0.75$.}
\label{f:angle}
\end{figure}


\subsection{Simulations of diffusion induced grain boundary motion} \label{digm}
We conclude our numerical results with two simulations of diffusion induced grain boundary motion. We consider the physical set-up of a film of metal, containing a single grain boundary. We denote the film by $A=\Omega\times[0,5]\subset \mathbb{R}^3$, with $\Omega = (-2,2)^2$, and we model the grain boundary by the surface $\Gamma(t)=\{(x,u(x,t))\,|\,x\in \Omega\}$. We impose the boundary condition
$$
\frac{\partial u}{\partial n}(x,t)=0~~~~\forall\,(x,t) \in \partial\Omega \times (0,T]
$$
such that the grain boundary meets the boundaries of the film orthogonally.
The film is immersed in a solute that diffuses into the grain boundary at the surfaces $x_1=\pm2$. We denote the concentration of the solute on the grain boundary by $w(x,t)\in[0,1]$, for $x\in \Omega$, and we assume that the solute concentration is set to $1$ on the surfaces $x_1=\pm2$ and satisfies zero flux boundary conditions at the surfaces $x_2=\pm 2$, i.e.
$$
w(x,t)=1~~\mbox{for}~x_1=\pm 2,~~\frac{\partial w}{\partial n}(x,t)=0~~\mbox{for}~x_2=\pm 2.
$$
We consider two initial configurations for the grain boundary, in the first we take the grain boundary to be the planar surface $x_3=1$ such that $u^0(x)=1$, while in the second we take 
\begin{equation}
u^0(x_1,x_2)=\left\{\begin{array}{cll}
1+\varepsilon&\mbox{if}&\,x_1 > \frac{\pi\varepsilon}{2} \\
\varepsilon \sin\left(\frac{x_1}{\varepsilon}\right)&\mbox{if}&|x_1| \leq   \frac{\pi\varepsilon}{2} \\
1-\varepsilon&\mbox{if}&\,x_1 < -  \frac{\pi\varepsilon}{2}\end{array}\right.
\label{eq:ic2}
\end{equation}
with $\varepsilon = 0.4$. For both configurations we assume that the concentration of solute on the grain boundary is initially zero, such that $w^0(x)=0$ for $x\in \Omega$.
In this set-up physically meaningful choices for $f(w)$ and $g(V,w)$ are $f(w)=w^2$ and $g(V,w)=|V|\,w$. 
Figure \ref{f:digm1} displays the solute concentration, $w_h^m(x)$, plotted on the grain boundary, $\Gamma_h^m=\{(x,u_h^m(x))\,|\,x\in \Omega\}$, at times $t^m=0,0.1,0.2,0.3$. In addition in each plot we display the initial grain boundary, depicted by the blue surface, and the outline of the metallic film $A=\Omega\times[0,5]$. The symmetry of this set-up  makes it equatable to the two-dimensional configurations studied in \cite{DES01} and \cite{SVY}. In particular we see a travelling wave solution comparable to the ones displayed in Figures 9 and 10 of \cite{DES01} and Figure 4.4 of \cite{SVY}. In Figure \ref{f:digm2} the initial surface is defined by (\ref{eq:ic2}) which gives rise to a fully three-dimensional simulation. We display the solute concentration, 
$w^m_h(x)$, plotted on the grain boundary, at times $t^m=0,0.2,0.4,0.6$, together with the initial grain boundary and the outline of the film. 


\begin{figure}[h]
\centering
$t=0$\hspace{5cm} $t=0.1$\\
\subfigure{\includegraphics[width = 0.4\textwidth]{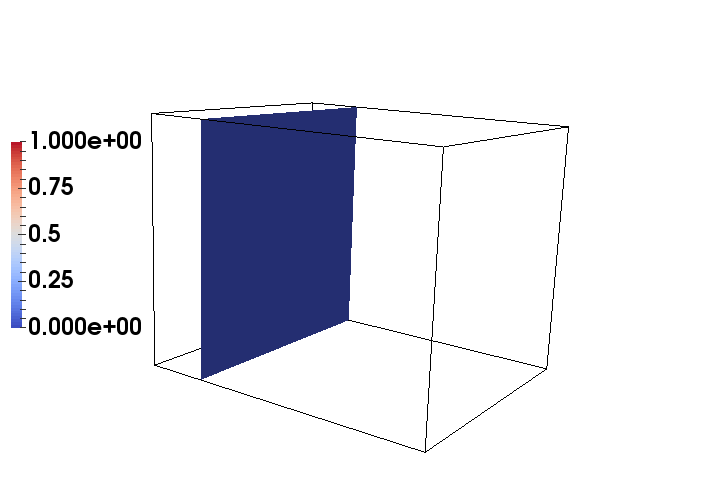}} 
\subfigure{\includegraphics[width = 0.4\textwidth]{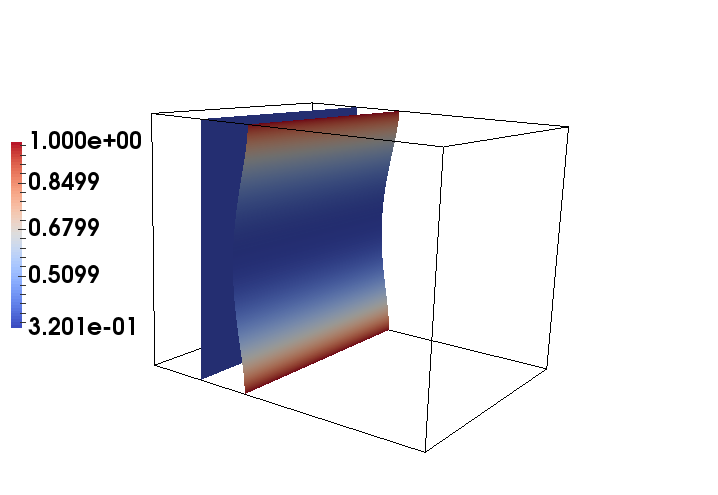}} \\
$t=0.2$\hspace{5cm} $t=0.3$\\
\subfigure{\includegraphics[width = 0.4\textwidth]{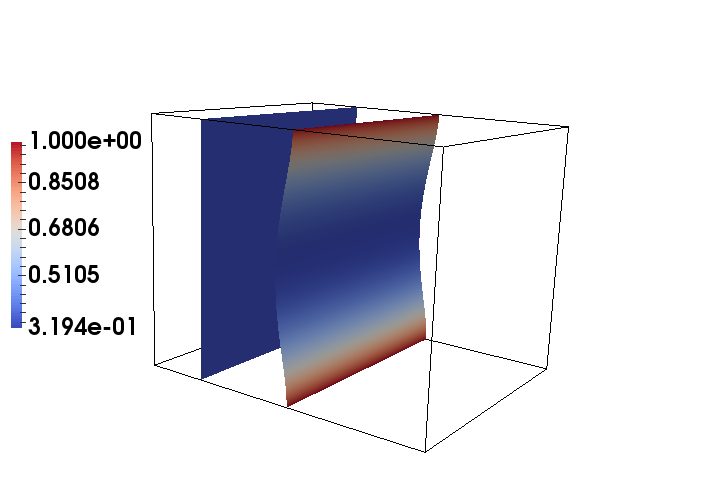}} 
\subfigure{\includegraphics[width = 0.4\textwidth]{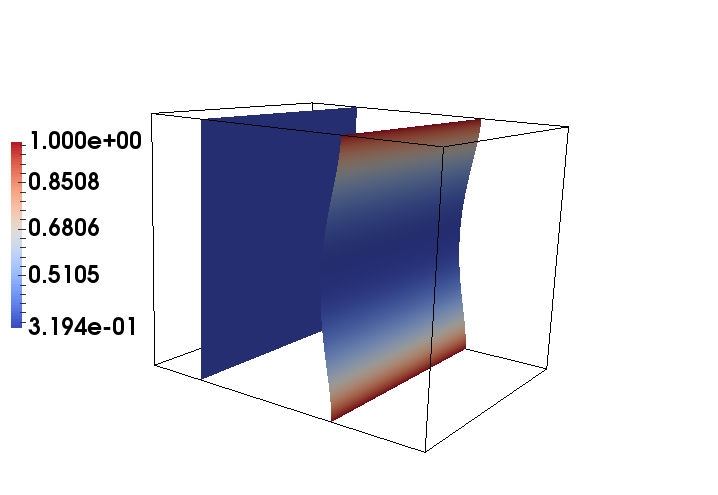}} \\
\caption{Travelling wave solution showing the grain boundary with the solute concentration at  $t^m=0,0.1,0.2,0.3$, with $u_h^0\equiv1$ and $w_h^0 \equiv 0$.}
\label{f:digm1}
\end{figure}

\begin{figure}[h]
\centering
$t=0$\hspace{5cm} $t=0.2$\\
\subfigure{\includegraphics[width = 0.4\textwidth]{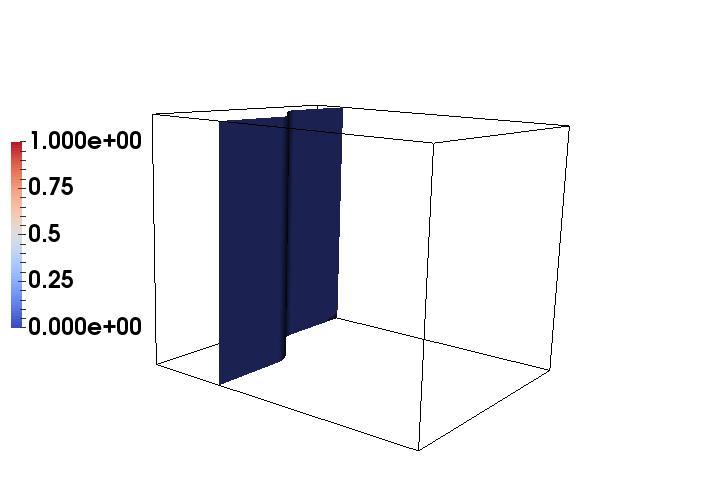}} 
\subfigure{\includegraphics[width = 0.4\textwidth]{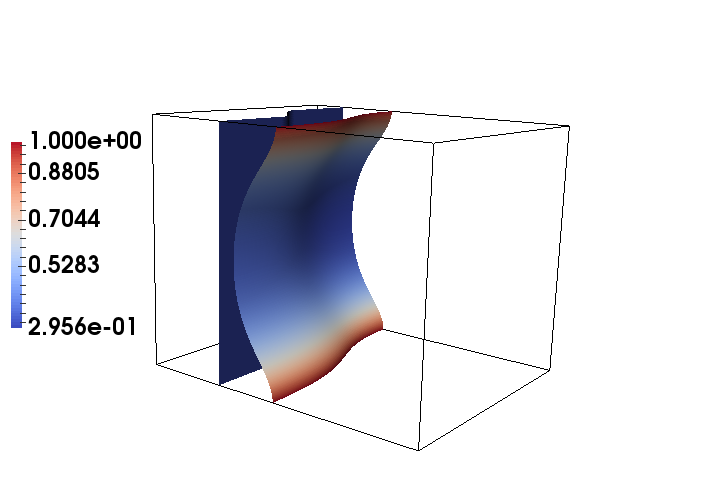}} \\
$t=0.4$\hspace{5cm} $t=0.6$\\
\subfigure{\includegraphics[width = 0.4\textwidth]{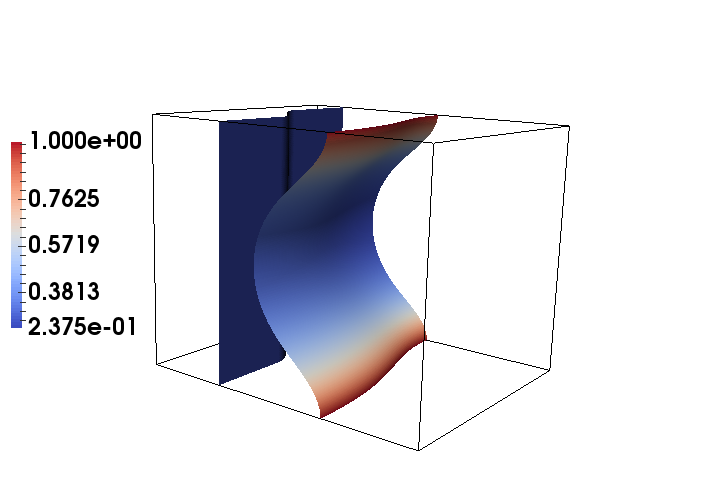}} 
\subfigure{\includegraphics[width = 0.4\textwidth]{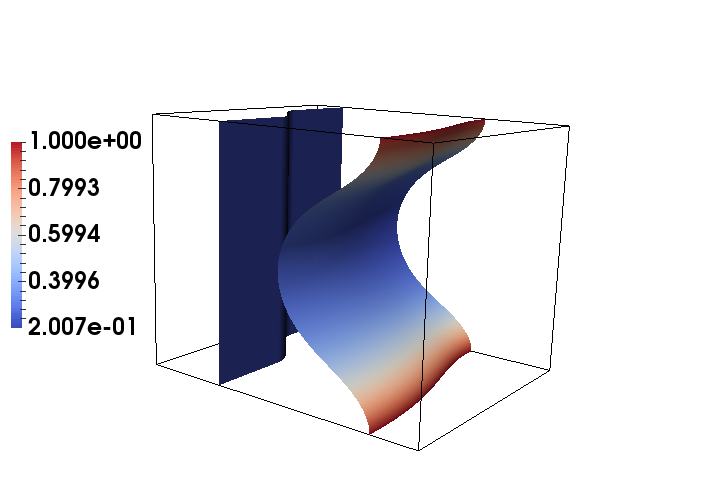}} 
\caption{Evolving grain boundary with the solute concentration at $t^m=0,0.2,0.4,0.6$, with the initial surface defined by (\ref{eq:ic2}) and $w_h^0 \equiv 0$.}
\label{f:digm2}
\end{figure}


\bibliographystyle{plain}


\end{document}